\documentclass[11pt,a4paper]{amsart}
\usepackage{amsmath}
\usepackage[psamsfonts]{amssymb}
\usepackage[all]{xy}
\usepackage{graphicx}
\usepackage[T1]{fontenc}
\usepackage{mathabx}
\usepackage{color}
\usepackage{a4wide}
\usepackage[colorlinks=true]{hyperref}
\usepackage{tikz-cd}
\usepackage{wasysym, stackengine, makebox}
\usetikzlibrary{cd}
\usepackage{mathtools}
\hypersetup{
     colorlinks   = true,
     citecolor    = blue
}

\newtheorem{thm}{Theorem}
\numberwithin{thm}{section}
\numberwithin{equation}{section}
\newtheorem{theorem}[thm]{Theorem}
\newtheorem*{theorem*}{Theorem}

\newtheorem{corollary}[thm]{Corollary}
\newtheorem*{corollary*}{Corollary}
\newtheorem{lemma}[thm]{Lemma}
\newtheorem{lemma-def}[thm]{Lemma-Definition}

\newtheorem{prop}[thm]{Proposition}
\newtheorem{proposition}[thm]{Proposition}

\newtheorem*{conjecture*}{Conjecture}

\newtheorem*{question*}{Question}

\newtheorem*{definition*}{Definition}

\newtheorem*{definitions*}{Definitions}

\newtheorem*{rem*}{Remark}
\theoremstyle{remark}
\newtheorem{remark}[thm]{Remark}

\newtheorem*{remark*}{Remark}

\newtheorem*{remarks*}{Remarks}
\newtheorem*{example*}{Example}

\newtheorem*{examples*}{Examples}

\newcommand{\R}{\mathbb{R}}
\newcommand{\Z}{\mathbb{Z}}
\newcommand{\Q}{\mathbb{Q}}
\newcommand{\C}{\mathbb{C}}
\newcommand{\N}{\mathbb{N}}
\newcommand{\K}{\mathbb K}

\def\CP{{\mathbb C}P}

\newcommand{\Id}{\mathit{Id}}

\newcommand{\ep}{\epilon}
\newcommand{\ga}{\gamma}

\newcommand{\Sp}{\text{Sp}}

\def\vr{\varphi}
\def\om{\omega}

\def\ep{\epsilon}

\def\Om{\Omega}
\DeclareMathOperator{\lcm}{lcm}

\def\M{{\widetilde M}}
\def\NN{{\widetilde N}}
\def\LL{{\Lambda}}
\def\cz{\mu}

\def\HC{{\mathrm{HC}}}
\def\HF{{\mathrm{HF}}}
\def\CF{{\mathrm{CF}}}
\def\EHF{{\mathrm{HF}}}
\def\SH{{\mathrm{SH}}}
\def\ESH{{\mathrm{SH}}}
\def\CC{{\mathrm{CC}}}
\def\Wc{{W_{\text{cone}}}}
\def\W{{\widehat W}}
\def\F{{\widetilde F}}
\def\bM{{\widebar M}}
\def\bxi{{\widebar\xi}}
\def\balpha{{\widebar\alpha}}

\begin{document}

\title[Two closed orbits for non-degenerate Reeb flows]{Two closed orbits for non-degenerate Reeb flows}

\author[M.~Abreu]{Miguel Abreu}
\address{Center for Mathematical Analysis, Geometry and Dynamical Systems,
Instituto Superior T\'ecnico, Universidade de Lisboa, 
Av. Rovisco Pais, 1049-001 Lisboa, Portugal}
\email{mabreu@math.tecnico.ulisboa.pt, macarini@math.tecnico.ulisboa.pt}
 
\author[J.~Gutt]{Jean Gutt}
\address{Mathematisches Institut, Universit\"at zu K\"oln, Weyertal 86-90, D-50931 K\"oln, Germany}
\email{gutt@math.uni-koeln.de}

\author[J.~Kang]{Jungsoo Kang}
\address{Department of Mathematical Sciences, Research institute in Mathematics, Seoul National University, Gwanak-Gu, Seoul 08826, South Korea}
\email{jungsoo.kang@snu.ac.kr}

\author[L.~Macarini]{Leonardo Macarini}
\address{Universidade Federal do Rio de Janeiro, Instituto de Matem\'atica,
Cidade Universit\'aria, CEP 21941-909 - Rio de Janeiro - Brazil}
\email{leomacarini@gmail.com}

\thanks{MA and LM were partially funded by FCT/Portugal through UID/MAT/04459/2013 and 
project PTDC/MAT-PUR/29447/2017. JK was supported by Samsung
Science and Technology Foundation (SSTF-BA1801-01). JG is supported by the SFB/TRR
191 ``Symplectic structures in Geometry, Algebra and Dynamics''. LM was partially supported by CNPq, Brazil.}

\subjclass[2010]{53D40, 53D25, 37J10, 37J55} \keywords{Closed orbits,
  Conley-Zehnder index, Reeb flows, equivariant symplectic homology}


\begin{abstract}
We prove that every non-degenerate Reeb flow on a closed contact manifold $M$ admitting a strong symplectic filling $W$ with vanishing first Chern class carries at least two geometrically distinct closed orbits provided that the positive equivariant symplectic homology of $W$ satisfies a mild condition. Under further assumptions, we establish the existence of two geometrically distinct closed orbits on any contact finite quotient of $M$. Several examples of such contact manifolds are provided, like displaceable ones, unit cosphere bundles, prequantization circle bundles, Brieskorn spheres and toric contact manifolds. We also show that this condition on the equivariant symplectic homology is preserved by boundary connected sums of Liouville domains. As a byproduct of one of our applications, we prove a sort of Lusternik-Fet theorem for Reeb flows on the unit cosphere bundle of not rationally aspherical manifolds satisfying suitable additional assumptions.
\end{abstract}

\maketitle
\section{Introduction}
A contact manifold is an odd-dimensional manifold $M^{2n+1}$ endowed with a contact structure, i.e.\ a codimension one distribution $\xi$ having a maximal non-integrability property. If we write locally the distribution as the kernel of a $1$-form, $\xi=\ker \alpha$, the condition is that $\alpha\wedge (d\alpha)^{n}$ is nowhere vanishing; such a $1$-form $\alpha$ is called a contact form. Throughout this paper we will tacitly assume that a contact structure is co-oriented, that is, $\alpha$ is defined globally. Associated to a contact form $\alpha$ we have the corresponding Reeb vector field $R_\alpha$ uniquely characterized by the equations $\iota_{R_\alpha}d\alpha=0$ and $\alpha(R_\alpha)=1$.

A central question in contact geometry is the Weinstein conjecture which states that every contact form on a closed manifold carries at least one closed Reeb orbit. The Weinstein conjecture was proved in dimension three by Taubes in 2007 \cite{Tau}. Taubes' result was later improved by Cristofaro-Gardiner and Hutchings \cite{GH} who proved that every contact form on a closed manifold of dimension three carries at least two simple closed Reeb orbits (this result for the tight sphere $S^3$ was proved independently in \cite{GHHM}; see also \cite{LL} for another alternative proof using a result from \cite{GHHM}). Unfortunately, these results hold only in dimension three and the situation in higher dimensions is wide open.

Suppose that $M$ admits a strong symplectic filling $(W,\om)$ such that $c_1(TW)=0$. For the sake of simplicity, assume, for a moment, that $W$ is a Liouville domain, that is, $\om=d\lambda$ is exact and the dual vector field of $\lambda$ with respect to $\om$ points outwards along $\partial W$. Then, one can associate to $W$ its positive $S^1$-equivariant symplectic homology $\HC_*(W)$ with an integer grading. It was introduced by Viterbo \cite{Vit} and developed by Bourgeois and Oancea \cite{BO10, BO13a, BO13b, BO17}. (For the sake of brevity and to emphasize the analogy with contact homology, we denote this homology by $\HC_*(W)$ rather than the usual notation $\SH^{+,S^1}_*(W)$ used in \cite{BO10, BO13a, BO13b, BO17}.) The positive equivariant symplectic homology has a filtration given by the free homotopy classes of loops in $W$ and given a set $\Gamma$ of free homotopy classes of loops we denote by $\HC_*^\Gamma(W)$ the corresponding homology.

Given a non-degenerate contact form $\alpha$ on $M$, it turns out that $\HC_*^\Gamma(W)$ can be obtained as the homology of a chain complex generated by the good closed Reeb orbits of $\alpha$ with free homotopy class in $\Gamma$; see \cite{GG,GGM2,GU}. In particular, the Weinstein conjecture holds whenever $\HC_*(W)$ does not vanish.

Assuming that $\HC_*(W)$ does not vanish, it is natural to ask if we can get more than one closed orbit. This is the multiplicity problem for closed Reeb orbits and the fundamental difficulty here is how to distinguish simple orbits from iterated ones. This difficulty already manifests itself in the classical problem of the multiplicity of closed geodesics on Riemannian manifolds; c.f. \cite{Kli}.

If $\HC_*(W)$ is asymptotically unbounded, that is, there exists a sequence $k_i \to \infty$ such that $\lim_{i\to\infty} \dim \HC_{k_i}(W) = \infty$ then it was proved in \cite{HM,McL} that every Reeb flow on $M$ carries infinitely many simple closed orbits. (These results were inspired by the corresponding result for geodesic flows due to Gromoll and Meyer \cite{GM}.) Thus, it is natural to ask what happens when $\HC_*(W)$ is \emph{not} asymptotically unbounded.

In this case, the situation is more delicate since we do have several examples of contact forms with finitely many simple closed orbits. Some works provide lower bounds for the number of simple closed orbits (some of these estimates are sharp), but, in general, these bounds request certain index conditions on the closed Reeb orbits; see \cite{AM2,DLLW,GG,GGM2,GK} and references therein.

The main result in this work establishes that a contact form $\alpha$ on $M$ has at least \emph{two} simple closed orbits provided that $\HC_*(W)$ satisfies a certain mild condition and assuming \emph{only} the non-degeneracy of $\alpha$. More precisely, we have the following statement. Throughout this paper, we will use the convention that the natural numbers are given by the positive integers. Moreover, the grading of $\HC_*(W)$ is given by the Conley-Zehnder index. Under the assumption that $c_1(TW)=0$, this grading is well defined up to a choice of the homotopy class of a non-vanishing section of the determinant line bundle $\Lambda_\C^{n+1}TW$; see \cite[Section 3]{AM3}, \cite{Es}, \cite{McL2} and \cite[Appendix C]{MR}. For a shorter notation, we will omit this dependence.

\begin{theorem}\label{thm:main}
Let $(M^{2n+1},\xi)$ be a closed contact manifold admitting a strong symplectic filling $W$ such that $c_1(TW)=0$. Let $\Gamma$ be a set of free homotopy classes of loops in $W$ closed under iterations and assume that there exist $K \in \N$ and a non-vanishing section $\sigma$ of the determinant line bundle $\Lambda_\C^{n+1}TW$ such that
\[
\dim \HC_{n}^\Gamma(W) < \dim \HC_{n+jK}^\Gamma(W)
\]
or
\[
\dim \HC_{-n}^\Gamma(W) < \dim \HC_{-n-jK}^\Gamma(W)
\]
for every $j \in \N$, where the grading in $\HC^\Gamma_*(W)$ is taken with respect to the homotopy class of $\sigma$. Then every non-degenerate Reeb flow on $M$ carries either infinitely many geometrically distinct closed Reeb orbits or at least two geometrically distinct closed Reeb orbits $\ga_1$ and $\ga_2$ such that their Conley-Zehnder indices satisfy $\cz(\ga_1^k)\neq \cz(\ga_2^k)$ for some $k \in \N$. Moreover, all these orbits have free homotopy class in $\Gamma$.
\end{theorem}

\begin{remark}
\label{rmk:contractible}
When $\Gamma=\{0\}$, i.e. when looking at contractible orbits, it is enough to assume that $c_1(TW)|_{\pi_2(W)}=0$. Moreover, in this case the grading of $\HC^\Gamma_*(W)$ does not depend on the choice of the homotopy class of $\sigma$.
\end{remark}

\begin{remark}
Note that we are \emph{not} assuming that $W$ is a Liouville domain. When $\om$ is not exact, $\HC_*^\Gamma(W)$ is a vector space over the universal Novikov field and we have to use an action filtration introduced by McLean and Ritter in \cite{MR}; see Section \ref{sec:ESH}. Using this action filtration we also have that, given a non-degenerate contact form $\alpha$ on $M$, $\HC_*^\Gamma(W)$ is the homology of a chain complex generated by the good closed Reeb orbits of $\alpha$ with free homotopy class in $\Gamma$.
\end{remark}

Let $M \xrightarrow{\tau} \bM$ be a finite covering. Suppose that $M$ carries a contact structure $\xi$ preserved by the deck transformations so that $\bM$ has an induced contact structure $\bxi$ such that $\tau^*\bxi=\xi$. We say in this case that $\bM$ is a \emph{contact finite quotient} of $M$. Assume that $(M,\xi)$ satisfies the hypotheses of Theorem \ref{thm:main} and that $H^1(M,\R)=0$. Given a non-degenerate contact form $\balpha$ on $\bM$ we can consider its lift to $M$ and applying Theorem \ref{thm:main} we can deduce that $\balpha$ has  at least two simple closed orbits; see Section \ref{sec:quotient}. Thus, we have the following corollary.

\begin{corollary}
\label{cor:quotient}
Let $(M,\xi)$ be a contact manifold satisfying the hypotheses of Theorem \ref{thm:main} and $H^1(M,\R)=0$. Then every non-degenerate Reeb flow on a contact finite quotient of $M$ carries at least two simple closed orbits. Moreover, the closed lifts of iterates of these orbits to $M$ have free homotopy class in $W$ contained in $\Gamma$.
\end{corollary}

\begin{remark}
The hypothesis that $H^1(M,\R)=0$ can be dropped if the first Chern class of the contact structure on the quotient vanishes; see Remark \ref{rmk:quotient}.
\end{remark}

\begin{remark}
When $\Gamma=\{0\}$, the hypothesis that $H^1(M,\R)=0$ can be replaced by the assumption that the map $\pi_1(M) \to \pi_1(W)$ induced by the inclusion is injective.
\end{remark}

\begin{remark}
\label{rmk:quotient2}
Since the deck transformations do not necessarily extend to the filling $W$, it is not clear that when $\Gamma=\{0\}$ the assumption that $c_1(TW)|_{\pi_2(W)}=0$ is enough to conclude the previous corollary, unless the map $\pi_1(M) \to \pi_1(W)$ induced by the inclusion is injective; c.f. Remark \ref{rmk:contractible}.
\end{remark}

In what follows, we will provide several examples of contact manifolds satisfying the hypotheses of Theorem \ref{thm:main}. The simplest one is the standard contact sphere $S^{2n+1}$ with its obvious filling $W$ given by the ball $D^{2n+2}$ whose positive equivariant symplectic homology is
\[
\HC_j(W) \cong
\begin{cases}
\Q & \text{if $j= n+2k$ for some $k \in \N$} \\
0 & \text{otherwise.}
\end{cases} 
\]
The existence of at least two simple closed orbits for every non-degenerate contact form on the standard contact sphere is well known; see, for instance, \cite{Gur,Ka}. We will furnish a very large class of examples, for which the sphere is a \emph{very} particular case.

To the best of our knowledge, all the examples known so far of closed contact manifolds admitting contact forms with finitely many simple closed orbits are prequantization circle bundles over symplectic orbifolds admitting Hamiltonian circle actions with isolated fixed points. Our list below covers several examples of such manifolds; see Corollaries \ref{cor:toric}, \ref{cor:prequantization} and \ref{cor:brieskorn}. This raises the following question.

\vskip .2cm
\noindent {\bf Question:} Let $M$ be a prequantization circle bundle over a closed symplectic orbifold admitting a Hamiltonian circle action with isolated fixed points. Suppose that $M$ admits a strong symplectic filling $W$ such that $c_1(TW)=0$. Is it true that $W$ satisfies the hypothesis of Theorem \ref{thm:main} for some $\Gamma$ and $\sigma$?
\vskip .2cm

Note that there are examples of contact manifolds admitting fillings with vanishing first Chern class that do not satisfy the hypothesis of Theorem \ref{thm:main}, at least for a naturally chosen non-vanishing section of the determinant line bundle. For instance, consider the cosphere bundle $S^*N$ of an orientable closed manifold admitting a Riemannian metric $g$ with negative sectional curvature and take the symplectic filling $D^*N$ given by the unit disk bundle. We have that $c_1(TD^*N)=0$ because $N$ is orientable. Moreover, the choice of a volume form in the base naturally induces a non-vanishing section of $\Lambda^{n+1}_\C TD^*N$ so that the Conley-Zehnder index of every non-degenerate closed geodesic coincides with its Morse index; see \cite[Section 1.4.5]{Ab}. Since the geodesic flow of $g$ is non-degenerate and the index of every closed geodesic vanishes, we have that $\HC_*(D^*N)$ is non-trivial only in degree zero and therefore it does not satisfy the hypothesis of Theorem \ref{thm:main}. However, every contact form on $S^*N$ supporting the standard contact structure has infinitely many simple closed orbits; see \cite{MPa}.

\subsubsection*{\bf Displaceable contact manifolds}\hfill\\[-2ex]

Given a contact manifold $(M,\xi)$ and an exact symplectic manifold $(X,d\lambda)$, an embedding $M \hookrightarrow X$ is called an exact contact embedding if it is bounding and there exists a contact form $\alpha$ supporting $\xi$ such that $\alpha-\lambda|_M$ is exact. Here bounding means that $M$ separates $X$ into two connected components, with one of them relatively compact. This embedding is displaceable if $M$ can be displaced from itself by a Hamiltonian diffeomorphism with compact support on $X$. We say that $X$ is convex at infinity if there exists an exhaustion $X = \cup_k X_k$ by compact subsets $X_k \subset X_{k+1}$ with smooth boundaries such that $\lambda|_{\partial X_k}$ is a contact form for every $k$. A big class of contact manifolds admitting displaceable exact contact embeddings in exact symplectic manifolds that are convex at infinity is given in \cite{BC02}: the boundary of every subcritical Stein manifold.

Let $(M,\xi)$ be a contact manifold admitting a displaceable exact contact embedding into a convex at infinity exact symplectic manifold $X$ such that $c_1(TX)|_{\pi_2(X)}=0$ and denote by $W$ the compact region in $X$ bounded by $M$. In Section \ref{sec:displaceable} we show that $W$ satisfies the hypothesis of Theorem \ref{thm:main} for $\Gamma=\{0\}$. Hence, we get the following result; see Remarks \ref{rmk:contractible} and \ref{rmk:quotient2}.

\begin{corollary}
\label{cor:displaceable}
Let $(M,\xi)$ be a contact manifold admitting a displaceable exact contact embedding into a convex at infinity exact symplectic manifold $X$ such that $c_1(TX)|_{\pi_2(X)}=0$ and denote by $W$ the compact region in $X$ bounded by $M$. Then every non-degenerate Reeb flow on $M$ has at least two simple closed orbits contractible in $W$. If $c_1(TX)=0$ and $H^1(M,\R)=0$ then every Reeb flow on a contact finite quotient of $M$ carries at least two simple closed orbits. Moreover, the closed lifts of iterates of these orbits to $M$ are contractible in $W$.
\end{corollary}

\subsubsection*{\bf Cosphere bundles and closed geodesics}\hfill\\[-2ex]

Let $N$ be a closed Riemannian manifold and $\LL N$ its free loop space. There is an isomorphism between the (non-equivariant) symplectic homology of $T^* N$ and the homology of $\LL N$ twisted by a local system of coefficients. For the $S^1$-equivariant version, if $N$ is orientable and spin, we have the isomorphism
\begin{equation}
\label{eq:CL}
\HC_{*}(D^* N) \cong H_*(\LL N/S^1, N;\Q),
\end{equation}
where $N \subset \LL N/S^1$ indicates the subset of constant loops, $D^*N$ is the obvious filling of the cosphere bundle $S^*N$ given by the unit disk bundle and we are taking the grading of $\HC_{*}(D^* N)$ given by a non-vanishing section of $\Lambda^{n+1}_\C TD^*N$ induced from the choice of a volume form in the base so that the Conley-Zehnder index of a non-degenerate closed geodesic coincides with its Morse index, see e.g.~\cite{BO17}. This isomorphism respects the filtration given by the free homotopy classes, that is,
\begin{equation}
\label{eq:CL filtered}
\HC_{*}^\Gamma(D^*N) \cong H_*(\LL^\Gamma N/S^1, N;\Q)
\end{equation}
for every set $\Gamma$ of free homotopy classes in $D^*N$, where $\LL^\Gamma N$ denotes the set of loops in $N$ with free homotopy class in $\Gamma$. (Note that, since $\pi_1(D^*N) \cong \pi_1(N)$, the set of free homotopy classes in $D^*N$ and $N$ are naturally identified. Moreover, $N\subset \LL^0 N/S^1$ and therefore if $\Gamma$ does not contain the trivial free homotopy class then the right hand side of the isomorphism \eqref{eq:CL filtered} has to be understood as $H_*(\LL^\Gamma N/S^1;\Q)$.) For general $N$, it is expected that the same isomorphism holds with a local system of coefficients as in the non-equivariant case but a rigorous proof has not been written in the literature yet. 

It turns out that if $N$ is simply connected and $H_*(\LL N/S^1, N;\Q)$ is not asymptotically unbounded then it satisfies the assumption in Theorem \ref{thm:main}. Using this, we can prove the following result. The proof is given in Section \ref{sec:cosphere}. Before we state it, let us recall a definition and introduce a notation. A topological space $X$ is $k$-simple if $\pi_1(X)$ acts trivially on $\pi_k(X)$. If a closed manifold $N$  has dimension bigger than one and $\pi_1(N) \cong \Z$ then $N$ is not rationally aspherical, that is, there exists $j>1$ such that $\pi_j(N) \otimes \Q \neq 0$; see Section \ref{sec:cosphere}. Let $k$ be the smallest such $j$. In what follows, $\xi_{can}$ denotes the canonical contact structure on $S^*N$.

\begin{corollary}
\label{cor:cosphere}
Let $N$ be a closed oriented spin manifold with dimension bigger than one. Suppose that $N$ satisfies one of the following conditions:
\begin{itemize}
\item[(i)] $\pi_1(N)$ is finite;
\item[(ii)] $\pi_1(N) \cong \Z$, $\pi_2(N)=0$ and $N$ is $k$-simple, with $k$ as discussed above;
\item[(iii)] $\pi_1(N)$ is infinite and there is no $a\in\pi_1(N)$ such that every non-zero $b\in\pi_1(N)$ is conjugate to some power of $a$.
\end{itemize}
In case (i), we have that every non-degenerate contact form on $(S^* N,\xi_{can})$ has at least two simple closed orbits. Under hypothesis (ii) or (iii), we have two simple closed orbits for any contact form on $(S^* N,\xi_{can})$, without assuming that it is non-degenerate.
\end{corollary}

\begin{remark}
\label{rmk:spin}
The hypothesis that $N$ is oriented spin is used only to have the isomorphism \eqref{eq:CL filtered}. Possibly, it can be relaxed once we have this isomorphism with the relative homology of $(\LL N/S^1,N)$ twisted by a local system of coefficients.
\end{remark}

\begin{remark}
\label{rmk:hyp simple}
In case (ii), the hypothesis that $N$ is $k$-simple can be relaxed in the following way: let $a$ be a generator of $\pi_1(N)$ and denote by $A$ the linear map corresponding to the action of $a$ on $\pi_k(N) \otimes \Q$. Then it is enough that $\ker(A-\Id)\neq 0$; see Remark \ref{rmk:simple}. This hypothesis and the assumption that $\pi_2(N)=0$ are probably just technical but we do not know how to drop them so far.
\end{remark}

Theorem \ref{thm:main} is used to prove Corollary \ref{cor:cosphere} only under hypothesis (i). For hypotheses (ii) and (iii), we show the existence of two periodic orbits $\ga_1$ and $\ga_2$ such that no iterate of $\ga_1$ is freely homotopic to $\ga_2$. This is easy in case (iii) using \eqref{eq:CL filtered} but highly non-trivial in case (ii) where we show that $\HC^0_*(D^*N)\neq 0$ and $\HC^a_*(D^*N)\neq 0$ for some non-trivial homotopy class $a$. The proof in case (ii) actually shows the following result; see Remark \ref{rmk:LF}. It can be considered as a sort of Lusternik-Fet theorem for Reeb flows; see e.g. \cite{AB}.

\begin{theorem}
Let $N$ be a closed not rationally aspherical manifold. Suppose that $N$ is oriented spin, $\pi_1(N)$ is abelian, $\pi_2(N)=0$ and $N$ is $k$-simple, with $k$ as discussed above. Then every (possibly degenerate) Reeb flow on $S^*N$ carries a \emph{contractible} closed orbit. As a consequence, if, furthermore, $\pi_1(N)$ is infinite, then every Reeb flow on $S^*N$ has at least two simple closed orbits.
\end{theorem}

\begin{remark}
Similarly to Corollary \ref{cor:cosphere}, the hypothesis that $N$ is $k$-simple can be weakened; see Remark \ref{rmk:LF}.
\end{remark}

The hypothesis that $N$ is oriented spin and the second and third conditions in item (ii) can be dropped when we restrict ourselves to Reeb flows given by geodesic flows of Finsler metrics as follows. The proof of item (i) in Theorem \ref{thm:main} uses only the fact that, given a non-degenerate contact form $\alpha$ on $M$, $\HC_*^\Gamma(W)$ is the homology of a chain complex generated by the good periodic orbits of $\alpha$ with free homotopy class in $\Gamma$ graded by the Conley-Zehnder index; the nature of the differential is absolutely unessential. Let $F$ be a Finsler metric on $N$. It is well known that the closed geodesics of $F$ are the critical points of the corresponding energy functional defined on the free loop space. We will say that $F$ \emph{has only one prime closed geodesic} if either the corresponding geodesic flow has only one simple closed orbit or $F$ is reversible (i.e. $F(x,v)=F(x,-v)$ for every $(x,v) \in TN$) and its geodesic flow has only two simple periodic orbits (given by the lifts of a closed geodesic $\ga(t)$ and its reversed geodesic $\ga(-t)$).

It turns out that if $F$ is bumpy (i.e. its geodesic flow is non-degenerate) and has only one prime closed geodesic then $H_*(\LL N/S^1, N;\Q)$ is the homology of the chain complex generated by the good periodic orbits of the geodesic flow of $F$ with trivial differential. Using this fact we can prove the following result. In what follows, we say that $F$ \emph{has at least two prime closed geodesics} if it does not have only one prime closed geodesic in the sense above. (Note that every Finsler metric has at least one prime closed geodesic.) 

\begin{corollary}
\label{cor:geodesics}
Let $N$ be a closed manifold with dimension bigger than one. Suppose that $N$ satisfies one of the following conditions:
\begin{itemize}
\item[(i)] $\pi_1(N)$ is finite;
\item[(ii)] $\pi_1(N) \cong \Z$;
\item[(iii)] $\pi_1(N)$ is infinite and there is no $a\in\pi_1(N)$ such that every non-zero $b\in\pi_1(N)$ is conjugate to some power of $a$.
\end{itemize}
In case (i), we have that every bumpy Finsler metric $F$ on $N$ has at least two prime closed geodesics. Under hypothesis (ii) or (iii), we have two prime closed geodesics for any Finsler metric $F$ on $N$, without assuming that it is bumpy.
\end{corollary}

\begin{remark}
Our contribution in this corollary is that we find two closed geodesics when $N$ has finite fundamental group and $F$ is bumpy. The remaining cases can be covered by classical minimax methods. This is in contrast with Corollary \ref{cor:cosphere} for which these minimax methods are not available, making the proof of item (ii) much harder than in the case of geodesic flows.
\end{remark}

\begin{remark}
We are not aware of any example of $N$ which is excluded in the statement, see \cite[Section 5]{Tai}. For instance, if $\pi_1(N)$ is abelian, $N$ meets the hypothesis in Corollary \ref{cor:geodesics}.
\end{remark}

\subsubsection*{\bf Good toric contact manifolds}\hfill\\[-2ex]

Toric contact manifolds are the odd dimensional analogues of toric symplectic manifolds. They can be defined as contact manifolds of dimension $2n+1$ equipped with an effective Hamiltonian action of a torus of dimension $n+1$. Good toric contact manifolds of dimension three are $(S^3, \xi_{\rm st})$ and its finite quotients. Good toric contact manifolds of dimension greater than three are compact toric contact manifolds whose torus action is not free. These form the most important class of compact toric contact manifolds and can be classified by the associated moment cones, in the same way that Delzant's theorem classifies compact toric symplectic manifolds by the associated moment polytopes. We refer to \cite{Le} for details.

In \cite{AM1} we showed that on any good toric contact manifold $(M^{2n+1},\xi)$ such that $c_1(\xi)=0$, any non-degenerate toric contact form is even, that is, all contractible closed orbits of its Reeb flow have even contact homology degree, where the contact homology degree of a closed orbit $\ga$ is given by $\cz(\ga)+n-2$. (As proved in \cite{AM3}, this is also true for the non-contractible closed Reeb orbits.) Suppose that $M$ admits a symplectic filling $W$ with vanishing first Chern class. Then, as showed in \cite{AM1,AM3}, $\HC_*^0(W)$ can be computed in a purely combinatorial way in terms of the associated momentum cone. Using this computation, we show in Section \ref{sec:toric} that $W$ satisfies the hypothesis of Theorem \ref{thm:main} for $\Gamma=\{0\}$ and consequently we get the following result. Note that the fundamental group of every good toric contact manifold $M$ is finite and consequently $H^1(M,\R)=0$.

\begin{corollary}
\label{cor:toric}
Let $(M,\xi)$ be a good toric contact manifold admitting a strong symplectic filling $W$ such that $c_1(TW)=0$. Then every non-degenerate contact form on a contact finite quotient of $M$ carries at least two geometrically distinct contractible closed orbits.
\end{corollary}

\begin{remark}
It turns out that every good toric contact manifold $(M,\xi)$ in dimensions three and five such that $c_1(\xi)=0$ admits a (toric) filling with vanishing first Chern class \cite{AM3}.
\end{remark}

\subsubsection*{\bf Prequantization circle bundles over symplectic manifolds}\hfill\\[-2ex]

Let $(B^{2n},\om)$ be a closed integral symplectic manifold. Consider the prequantization circle bundle $(M,\xi)$ of $(B,\om)$, that is, the contact manifold given by the total space of a principal circle bundle over $B$ whose first Chern class is $-[\omega]$ and with contact structure given by the kernel of a connection form. Suppose that $M$ admits a symplectic filling $W$ with vanishing first Chern class. Then, under some assumptions on $B$, we can show that $W$ satisfies the hypothesis of Theorem \ref{thm:main} with $\Gamma=\{0\}$. More precisely, we have the following result. In what follows,
\[
c_B := \inf \{k \in \N \mid \exists S \in \pi_2(B) \text{ with }
\langle c_1(TB),S \rangle = k\} 
\]
denotes the minimal Chern number of $B$.

\begin{corollary}
\label{cor:prequantization}
Let $(M,\xi)$ be a prequantization circle bundle of a closed integral symplectic manifold $(B,\om)$ such that $\om|_{\pi_2(B)}\neq 0$, $c_1(TB)|_{\pi_2(B)}\neq 0$ and, furthermore, $H_{k}(B;\Q)=0$ for every odd $k$ or $c_B>n$. Suppose that $M$ admits a strong symplectic filling $W$ such that $c_1(TW)=0$. Then every non-degenerate Reeb flow on $M$ carries at least two geometrically distinct closed orbits contractible in $W$. If, additionally, $H^1(M,\R)=0$ then every contact form on a contact finite quotient of $M$ carries at least two geometrically distinct closed orbits. Moreover,  the closed lifts of iterates of these orbits to $M$ are contractible in $W$.
\end{corollary}

\begin{remark}
\label{rmk:Gysin}
It follows from the Gysin exact sequence that $H^1(M,\R)=0$ whenever $H_1(B;\Q)=0$.
\end{remark}

\begin{remark}
When $\om|_{\pi_2(B)}=0$ and $B$ satisfies some extra conditions (for instance, when $\pi_i(B)=0$ for every $i\geq 2$) it is proved in \cite{GGM1} (c.f. \cite{GS}) that every Reeb flow on $M$ (possibly degenerate) carries infinitely many simple closed orbits.
\end{remark}

\begin{remark}
We have that $H_{*}(B;\Q)$ vanishes in odd degrees and $c_1(TB)|_{\pi_2(B)}\neq 0$ whenever $B$ admits a Hamiltonian circle action with isolated fixed points.
\end{remark}

The proof of the previous corollary is given in Section \ref{sec:prequantization}. Now, note that the prequantization bundle $M$ has a natural symplectic filling $W$ given by the corresponding disk bundle in the complex line bundle $L \xrightarrow{\pi} B$ whose first Chern class is $-[\omega]$. Suppose that $B$ is monotone, that is, $[\om]=\lambda c_1(TB)$ for some $\lambda \in \R$. (We say that $B$ is positive monotone if $\lambda>0$.) One can check that
\[
c_1(TW)=(1-\lambda)\pi^*c_1(TB).
\]
Consequently, when $\lambda=1$ we have that $c_1(TW)=0$. Now, suppose that $\lambda$ is an integer bigger than one and let $\M$ be the prequantization bundle of $(B,\frac{1}{\lambda}\om)$. It is easy to see that $M$ is the finite quotient of $\M$ by the $\Z_\lambda$-action induced by the obvious $S^1$-action on $\M$. Thus, we have the following corollary; see Remark \ref{rmk:Gysin}.

\begin{corollary}
Let $(M,\xi)$ be the prequantization circle bundle of a closed integral symplectic manifold $(B,\om)$ such that $\om|_{\pi_2(B)}\neq 0$, $c_1(TB)|_{\pi_2(B)}\neq 0$ and, furthermore, $H_{k}(B;\Q)=0$ for every odd $k$ or $c_B>n$. Suppose that $[\om]=\lambda c_1(TB)$ for some $\lambda \in \N$ and that $H_1(B;\Q)=0$. Then every contact form on a contact finite quotient of $M$ carries at least two geometrically distinct closed orbits. Moreover, the closed lifts of iterates of these orbits to $M$ have contractible projections to $B$.
\end{corollary}

\begin{remark}
If a closed curve $\ga$ in $M$ has contractible projection to $B$ then it must be homotopic to some iterate of the fiber and therefore, since $\om|_{\pi_2(B)}\neq 0$, its homotopy class $[\ga] \in \pi_1(M)$ is nilpotent, i.e. $[\ga]^k=0$ for some $k$.
\end{remark}

\begin{remark}
Note that if $c_B=1$ then, since $\om$ has to be integral, $\lambda$ must be an integer. Therefore, $[\om]=\lambda c_1(TB)$ for some $\lambda \in \N$ whenever $B$ is positive monotone and $c_B=1$.
\end{remark}

\subsubsection*{\bf Brieskorn spheres}\hfill\\[-2ex]

Given $a=(a_0,\dots,a_{n+1}) \in \N^{n+2}$ define $\Sigma_a$ as the intersection of the hypersurface
\[
z_0^{a_0} + \dots + z_{n+1}^{a_{n+1}} = 0
\]
in $\C^{n+2}$ with the unit sphere $S^{2n+3} \subset \C^{n+2}$. It is well known that $\alpha_a=\frac{i}{8}\sum_{j=0}^{n+1} a_j(z_jd\bar z_j-\bar z_jdz_j)$ defines a contact form on $\Sigma_a$ and $(\Sigma_a,\xi_a:=\ker\alpha_a)$ is called a Brieskorn manifold. When $n$ is even, $a_0=\pm 1\ \text{mod}\ 8$ and $a_1=\dots=a_{n+1}=2$ we have that $\Sigma_a$ is diffeomorphic to the sphere $S^{2n+1}$ and called a Brieskorn sphere. Brieskorn spheres admit strong symplectic fillings given by Liouville domains $W$ satisfying $c_1(TW)=0$ and it turns out that $W$ satisfies the hypothesis of Theorem \ref{thm:main} with $\Gamma=\{0\}$; see Section \ref{sec:brieskorn}. Therefore, we obtain the following result which is a generalization of \cite[Theorem C]{Ka}.

\begin{corollary}
\label{cor:brieskorn}
Let $M$ be a contact finite quotient of a Brieskorn sphere. Then every non-degenerate Reeb flow on $M$ carries at least two geometrically distinct closed orbits.
\end{corollary}

\subsubsection*{\bf Connected sums}\hfill\\[-2ex]

Let $(W_1,\lambda_1)$ and $(W_2,\lambda_2)$ be two Liouville domains of dimension $2n+2$. The boundary connected sum of them is again a Liouville domain $(W_1\#W_2,\lambda_1\#\lambda_2)$ and the contact connected sum $(\partial W_1\#\partial W_2,\xi_1\#\xi_2)$ is the boundary of it. The following result establishes that the main hypothesis of Theorem \ref{thm:main} is preserved by boundary connected sums of Liouville domains, furnishing many other examples of contact manifolds satisfying the assumptions of Theorem \ref{thm:main}.

\begin{theorem}
\label{thm:sums}
Let $(W_1,\lambda_1)$ and $(W_2,\lambda_2)$ be Liouville domains of dimension $2n+2$ with vanishing first Chern class. Assume that there are non-vanishing sections $\sigma_1$ and $\sigma_2$ of $\Lambda^{n+1}_\C TW_1$ and $\Lambda^{n+1}_\C TW_2$ respectively satisfying the hypothesis of Theorem \ref{thm:main} with $\Gamma$ given by the set of all free homotopy classes. Suppose that $c_1(T(W_1\#W_2))=0$ and let $\sigma$ be a non-vanishing section of $\Lambda^{n+1}_\C T(W_1\#W_2)$ extending $\sigma_1$ and $\sigma_2$. Then $W_1\#W_2$ satisfies the hypothesis of Theorem \ref{thm:main} with the grading of $\HC_*(W_1\#W_2)$ induced by $\sigma$.
\end{theorem}

\vskip .2cm
\noindent {\bf Organization of the paper.} In Section \ref{sec:ESH} we discuss equivariant symplectic homology for symplectic fillings with vanishing first Chern class. The proof of Theorem \ref{thm:main} is established in Section \ref{sec:main}. Corollaries \ref{cor:quotient}, \ref{cor:displaceable}, \ref{cor:cosphere}, \ref{cor:geodesics}, \ref{cor:toric}, \ref{cor:prequantization} and \ref{cor:brieskorn} are proved in Sections \ref{sec:quotient}, \ref{sec:displaceable}, \ref{sec:cosphere}, \ref{sec:geodesics}, \ref{sec:toric}, \ref{sec:prequantization} and \ref{sec:brieskorn} respectively. Finally, Section \ref{sec:sums} is devoted to the proof of Theorem \ref{thm:sums}.

\subsection*{Acknowledgements}
We are grateful to Viktor Ginzburg, Marco Mazzucchelli and Otto van Koert for useful discussions and to Gustavo Granja for suggesting us the proof of Proposition \ref{prop:hurewicz} and explaining us the Whitehead-Serre theorem for nilpotent spaces.

\section{Equivariant symplectic homology for symplectic fillings with vanishing first Chern class}
\label{sec:ESH}

Let $(M^{2n+1},\xi)$ be a closed contact manifold admitting a strong symplectic filling $W$ such that $c_1(TW)=0$. Fix a non-vanishing section of the determinant line bundle $\Lambda^{n+1}_\C TW$. Then, as we will shortly explain in this section, one can associate to $W$ its positive equivariant symplectic homology $\HC_*(W)$. This construction, briefly discussed in Section \ref{sec:ESH2}, is well known when $W$ is a Liouville domain, but without this assumption we have to use coefficients in the universal Novikov field and an action filtration introduced by McLean and Ritter \cite{MR}; c.f. \cite[Section 2.2]{GS}.

In Section \ref{sec:MB}, we discuss Morse-Bott spectral sequences which play a key role in the computation of equivariant symplectic homology. Using these sequences, one can show that given a non-degenerate contact form $\alpha$ on $M$, $\HC_*(W)$ can be obtained as the homology of a chain complex generated by the good closed Reeb orbits of $\alpha$. In particular, this homology is a contact invariant whenever $(M,\xi)$ admits a non-degenerate contact form such that every periodic orbit has Conley-Zehnder index with the same parity or, more generally, the set of Conley-Zehnder indexes is lacunary, i.e. does not contain any two consecutive integers.

\subsection{Equivariant symplectic homology}
\label{sec:ESH2}

Suppose that the symplectic filling $W$ is a Liouville domain. The positive equivariant symplectic homology $\HC_*(W)$ of $W$ was introduced by Viterbo \cite{Vit} and developed by Bourgeois and Oancea \cite{BO10, BO13a, BO13b, BO17}. As noticed in the introduction, the equivariant symplectic homology has a filtration given by the free homotopy classes in $W$ and given a set $\Gamma$ of free homotopy classes we denote by $\HC_*^\Gamma(W)$ the corresponding homology.

It turns out that $\HC_*^\Gamma(W)$ can be obtained as the homology of a chain complex $\CC^\Gamma_*(\alpha)$ with rational coefficients generated by the good closed Reeb orbits of a non-degenerate contact form $\alpha$ on $M$ with free homotopy class in $\Gamma$ graded by the corresponding Conley-Zehnder indices; see \cite[Proposition 3.3]{GG} and \cite[Lemma 2.1]{GU}. (Recall that a non-degenerate closed Reeb orbit is good if its index has the same parity of the index of the underlying simple closed orbit; otherwise it is called bad.) In general, the differential in $\CC^\Gamma_*(\alpha)$ depends on the filling $W$ and several extra choices. However, if $M$ admits a contact form such that the Conley-Zehnder indices of every good closed orbit have the same parity then clearly, by degree reasons, the differential must vanish and consequently the positive equivariant symplectic homology is a contact invariant.

Now, suppose that $(W,\om)$ is a strong symplectic filling with vanishing first Chern class but not necessarily a Liouville domain (in particular, $\om$ does not need to be exact.). Then we can still associate to $W$ its positive equivariant symplectic homology as we will briefly explain below. The novelty in this case is that $\HC_*^\Gamma(W)$ is a vector space over the universal Novikov field and we have to use an action filtration introduced by McLean and Ritter in \cite{MR}. We will follow \cite[Appendixes D and F]{MR} closely.

Let $\alpha$ be a non-degenerate contact form on $\partial W$ such that $d\alpha=\om|_{\partial W}$. Consider the symplectic completion $\W$ of $W$ obtained by attaching to $W$ the symplectic cone $(\Wc:=[1,\infty)\times \partial W,d(r\alpha))$ where $r$ is the coordinate in $[1,\infty)$.

Let $Z=r\partial_r$ be the Liouville vector field defined on $\Wc$ and uniquely characterized by the equation $\iota_{Z}\om=r\alpha$. Denote by $Y$ the Reeb vector field on $\Wc$ defined by $\iota_Y d\alpha=0$ and $\alpha(Y)=1$ (clearly, here we are tacitly identifying $\partial W$ with $\{1\}\times \partial W$). In what follows, by \emph{Reeb periods} we mean the periods of the closed orbits of $Y$. Given $r_0 \geq 1$, a compatible almost complex structure $J$ on $\W$ is said to be of contact type for $r\geq r_0$ if $JZ=Y$ on $[r_0,\infty) \times \partial W$.

Let $J$ be a compatible almost complex structure on $\W$ and $H: \W \to \R$ a smooth Hamiltonian. We will use the sign convention that its Hamiltonian vector field $X_H$ is given by the equation $\iota_{X_H}\om=dH$. The pair $(H,J)$ is called admissible if there exists $r_0\geq 1$ such that
\begin{enumerate}
\item $J$ is of contact type for $r\geq r_0$,
\item $H|_{[r_0,\infty) \times \partial W}=h(r)$ only depends on the radial coordinate,
\item $h'(r_0)>0$ is smaller than the minimal Reeb period,
\item $h''(r_0)>0$ and $h''(r)\geq 0$ for $r\geq r_0$,
\item if $h''(r)=0$ for some $r\geq r_0$ then we require that $h'(r)$ is not a Reeb period,
\item for large $r$, $h'(r)$ is a constant not equal to a Reeb period,
\item for $r\leq r_0$, i.e. on $W \cup [1,r_0] \times \partial W$, $H$ is Morse and $C^2$-small so that all closed 1-orbits of $X_H$ in $r\leq r_0$ are the critical points of $H$.
\end{enumerate}

Note that in the definition of admissible Hamiltonians we are fixing the contact form $\alpha$ used to define the radial coordinate $r$ in $\Wc$. Moreover, by our aforementioned sign convention, the Hamiltonian vector field $X_H$ is a \emph{negative} multiple of the Reeb vector field for $r\geq r_0$.

With the conditions above, the necessary compactness results for Floer trajectories hold, and this allows us to define the Floer homology $\HF(H,J)$ using coefficients in the universal Novikov field
\[
\Lambda = \bigg\{\sum_{i=1}^\infty n_iT^{a_i}\ ;\, a_i \in \R,\, a_i\to\infty,\, n_i\in \K\bigg\},
\]
where $\K$ is any given field of characteristic zero. Moreover, this homology does not depend on the choice of $J$ satisfying property (1). (More precisely, we need to pick time-dependent perturbations $H_t$ and $J_t$ (1-periodic in time) such that $H_t$ is non-degenerate and we have transversality for the relevant spaces of Floer trajectories. Moreover, it is required that, for a sufficiently large $R$, $J_t$ is of contact type and $H_t$ is of the form $ar+b$ for $r\geq R$, where $a$ is not equal to any Reeb period. We call the corresponding pair $(H_t,J_t)$ regular.)

\begin{remark}
Due to our sign conventions, the grading of Floer homology is given by \emph{minus} the Conley-Zehnder index. This is in accordance with the fact that the Hamiltonian vector field is a negative multiple of the Reeb vector field for $r\geq r_0$.
\end{remark}
 
We define a partial ordering on the set of admissible Hamiltonians in the following way: $H_0 \leq H_1$ if $H_0(x) \leq H_1(x)$ for every $x\in \W$. (For time-dependent perturbations $H_{i,t}$ ($i=0,1$) we require that $H_{0,t}(x) \leq H_{1,t}(x)$ for  every $(x,t)\in \W \times S^1$.) An admissible homotopy is a family of pairs $(H_s,J_s)$ satisfying properties (1)-(7) above for a fixed $r_0$ such that $\partial_s h'_s \leq 0$, where $h_s$ is the function such that $H_s|_{[r_0,\infty) \times \partial W}=h_s(r)$. Given an admissible homotopy $(H_s,J_s)$ we can define the continuation morphisms $\HF(H_0,J_0) \to \HF(H_1,J_1)$. (As before, we actually have to take suitable time-dependent perturbations $H_{s,t}$ and $J_{s,t}$ chosen generically in order to achieve transversality for the relevant spaces of parametrized Floer trajectories.)

Using this, we can define the symplectic homology as
\[
\SH_*(W,\alpha):=\varinjlim \HF_*(H)
\]
where the direct limit is taken over the set of admissible Hamiltonians with the partial ordering and continuation morphisms defined above. It turns out that this homology does not depend on the choice of $\alpha$; see \cite[Theorem 8]{Rit1} and \cite[Section 2.2]{BR}.

Now, we will discuss an action filtration introduced by McLean and Ritter \cite[Appendix D]{MR}. Fix an admissible Hamiltonian $H: \W \to \R$ and consider a smooth non-decreasing cut-off function $\phi: \R \to [0,1]$ such that
\begin{enumerate}
\item $\phi=0$ for $r\leq r_0$ and $\phi>0$ for $r>r_0$,
\item $\phi'>0$ for $r_0<r<r_1$ and $h'(r)$ is constant for $r\geq r_1$,
\item $\phi(r)=1$ for large $r$.
\end{enumerate}
Consider the 2-form $\eta$ on $\W$ given by
\[
\eta=
\begin{cases}
d(\phi(r)\alpha)\ & \text{on}\  \Wc \\
0\ & \text{on}\  W.
\end{cases}
\]
Notice that, by property (1) of $\phi$, this form vanishes on $W \cup [1,r_0] \times \partial W$. We associate to $\eta$ the 1-form $\Omega_\eta$ on the free loop space $C^\infty(S^1,\W)$, with $S^1=\R/\Z$, given by
\[
(\Omega_\eta)_x(X)=\int_0^1 \eta(X(t),\dot x(t)-X_{H}(x(t)))\,dt
\]
for every $x \in C^\infty(S^1,\W)$ and $X \in T_xC^\infty(S^1,\W)$. Following \cite{MR}, we will introduce a primitive for $\Om_\eta$ which will give us our action filtration. Let $f: \R \to [0,\infty)$ be the smooth function given by
\[
f(r)=\int_0^r \phi'(\tau)h'(\tau)\,d\tau.
\]
Note that, by the properties of $h$ and $\phi$, it satisfies
\begin{enumerate}
\item $f=0$ for $r\leq r_0$ and $f>0$ for $r>r_0$,
\item $f$ is bounded.
\end{enumerate}
The \emph{filtration functional} $F: C^\infty(S^1,\W) \to \R$ is defined as
\[
F(x) = -\int_{S^1} x^*(\phi\alpha) - \int_0^1 f(r(x(t)))\,dt,
\]
where $r(x(t))$ is the $r$-coordinate of $x(t)$. (Note that $\phi$ and $f$ vanish on $W$ and consequently $F$ is well defined for any curve in $\W$.)

\begin{remark}
Notice that our action functional differs from that in \cite{MR} in the sign in the second term. Moreover, we are using the sign convention $\iota_{X_H}\om=dH$ rather than $\iota_{X_H}\om=-dH$ as in \cite{MR}. The reason is that we are using \emph{homological} conventions rather than cohomological ones as in \cite{MR}. It is easy to see that the arguments from \cite{MR} work in this case with the appropriate sign changes.
\end{remark}

Given an admissible homotopy $(H_s,J_s)$ such that $h_s'(r)$ is constant for $r\geq r_1$ and all $s$ (this constant depends on $s$), define $F_s$ as the corresponding filtration functional for $H_s$, that is, $F_s(x) = -\int_{S^1} x^*(\phi\alpha) - \int_0^1 f_s(r(x(t)))\,dt$ where $f_s(r)=\int_0^r \phi'(\tau)h_s'(\tau)\,d\tau$.

This action functional has the following nice properties. 

\begin{theorem}[Theorem 6.2 in \cite{MR}]
\label{thm:action}
The filtration functional $F$ satisfies
\begin{enumerate}
\item Exactness: $F$ is a primitive of $\Omega_\eta$.
\item Negativity: $dF(\partial_s u)\leq 0$ for any Floer trajectory $u: \R \times S^1 \to \W$ of an admissible $(H,J)$.
\item Separation: $F=0$ on all loops in $r\leq r_0$ and $F>0$ on the 1-orbits of $H$ in $r\geq r_0$.
\item Compatibility: $dF_s(\partial_s u)\leq 0$ for any parametrized Floer trajectory $u: \R \times S^1 \to \W$ of an admissible homotopy $(H_s,J_s)$ such that $h_s'(r)$ is constant for $r\geq r_1$ and all $s$.
\item Strictness: $F(x_-) > F(x_+)$ for any Floer trajectory joining distinct orbits $x_-$ and $x_+$ with $r(x_+)\geq r_0$.
\end{enumerate}
Thus, $F$ determines a filtration on the Floer chain complex of a given admissible pair $(H,J)$ and this filtration is respected by Floer continuation maps for admissible homotopies $(H_s,J_s)$.
\end{theorem}

Let $\CF_*(H,J)$ be the Floer complex of $(H,J)$, $\CF_*^0(H,J)$ the subcomplex generated by the 1-orbits with action $F\leq 0$ and $\CF_*^+(H,J):=\CF_*(H,J)/\CF_*^0(H,J)$. The homology of $\CF_*^+(H,J)$ is denoted by $\HF_*^+(H,J)$. Note that, in principle, this homology depends on the choice of the cut-off function $\phi$ but, as proved in \cite[Corollary 6.4]{MR}, actually it does not. Moreover, it does not depend on the choice of $J$ and the direct limit 
\[
\SH_*^+(W,\alpha):=\varinjlim \HF_*^+(H)
\]
is the positive symplectic homology of $(W,\alpha)$. As before, it turns out that this homology does not depend on the choice of $\alpha$ (c.f. \cite[Theorem 8]{Rit1}) and therefore will be denoted by $\SH_*^+(W)$.

Now, the construction of positive equivariant symplectic homology is analogous to the construction from \cite{BO10, BO13a, BO13b, BO17,Vit}. In what follows, we will very briefly explain this; we refer to \cite{BO10, BO13a, BO13b, BO17,KvK} for details. Given $N \in \N$, we have to consider $S^1$-invariant families of (non-degenerate time-dependent perturbations of) admissible Hamiltonians $H: \W \times S^{2N+1} \times S^1 \to \R$, where the $S^1$-action is the diagonal one on $S^{2N+1} \times S^1$, and the space of $S^1$-equivariant Floer trajectories, consisting of pairs $(u,z)$ given by maps $u: \R \times S^1 \to \W$ and $z:\R \to S^{2N+1}$ such that
\[
\partial_s u + J_{z(s),t}(u)(\partial_t u - X_{H_{z(s),t}}(u))=0,
\]
\[
\dot z(s) - \int_{S^1} \nabla_{z} H(u(s,t),z(s),t)\,dt=0
\]
and
\[
\lim_{s\to-\infty} (u(s,t),z(s))=(\ga_-(t),z_-),\quad \lim_{s\to+\infty} (u(s,t),z(s))=(\ga_+(t),z_+)
\]
uniformly in $t$, where $J_{z,t}$, with $(z,t) \in S^{2N+1}\times S^1$, is a parametrized family of $S^1$-invariant compatible almost complex structures, $\nabla_z$ stands for the gradient in the $z$-coordinate with respect to an $S^1$-invariant metric on $S^{2N+1}$, $z_\pm \in S^{2N+1}$ and $\ga_\pm$ are closed 1-orbits of $H_{z_\pm,t}$. Choosing $H$ and $J$ generically, we have the corresponding $S^1$-equivariant Floer homology $\EHF^{S^1,N}_*(H)$. Using parametrized versions of the previous equations, we can define continuation morphisms $\EHF^{S^1,N}_*(H_0) \to \EHF^{S^1,N}_*(H_1)$ and using these we can define, analogously to the non-equivariant case,
\[
\ESH_*^{S^1,N}(W,\alpha):=\varinjlim \EHF_*^{S^1,N}(H)
\]
where the direct limit is taken over the set of admissible $S^1$-invariant families of Hamiltonians with the partial ordering and continuation morphisms similar to the non-equivariant case. The $S^1$-equivariant symplectic homology of $W$ is defined as
\[
\ESH_*^{S^1}(W,\alpha):=\varinjlim_{N} \ESH_*^{S^1,N}(W,\alpha),
\]
where the direct limit is taken with respect to the embeddings $S^{2N+1} \hookrightarrow S^{2N+3}$, which induce maps $\ESH_*^{S^1,N}(W,\alpha) \to \ESH_*^{S^1,N+1}(W,\alpha)$. As in the non-equivariant case, $\ESH_*^{S^1}(W,\alpha)$ does not depend on the choice of $\alpha$, since the proof of \cite[Theorem 8]{Rit1} extends to the equivariant setup. Therefore, we will denote it by $\ESH_*^{S^1}(W)$. Finally, the positive $S^1$-equivariant symplectic homology $\HC_*(W):=\ESH_*^{+,S^1}(W)$ is defined in the same way as in the non-equivariant case, using the action filtration $F$ defined above.

\subsection{Morse-Bott spectral sequence}
\label{sec:MB}
Assume that the contact form $\alpha$ is non-degenerate or, more generally, Morse-Bott. If $\alpha$ is degenerate, assume, furthermore, that the linearized Reeb flow is complex linear with respect to a unitary trivialization of the contact structure along every periodic orbit. This last condition is satisfied, for instance, when the Reeb flow of $\alpha$ generates a free circle action; see \cite[Remark 8.8]{KvK}. Under this assumption, we have a Morse-Bott spectral sequence that is very useful to compute $\ESH_*^{S^1}(W)$ and $\ESH_*^{+,S^1}(W)$ as we will very briefly explain in this section.

Given an admissible Hamiltonian $H$ as in the previous section, note that the non-constant closed 1-orbits of $H$ have radius $r_1<r_2<\dots<r_{p_{max}}$ where the $r_p$'s are such that $h'(r_p)$ is a Reeb period of $\alpha$ (by our assumptions on $H$ there are finitely many such radii). Take $G: \Z \to \R$ such that
\begin{enumerate}
\item $G$ is increasing (not necessarily strictly),
\item $G(p)=p$ for $p\leq 0$,
\item $G(p)=\phi(r_p)h'(r_p)-f(r_p)$ for $p\in \{1,\dots,p_{max}\}$.
\item $\lim_{p\to\infty} G(p)=\infty$.
\end{enumerate}
Note here that $\phi(r)h'(r)-f(r)=0$ for $r\leq r_0$ and $(\phi h'-f)'(r)=\phi(r)h''(r)\geq 0$ for every $r$ and therefore the existence of $G$ is clear. Note also that $F(x)=G(p)$ for every closed 1-orbit $x$ of $H$ such that $r(x)=r_p$ and $F(x)=0$ for every constant orbit $x$.

Define $A(H)=\{F(x);\,x\text{ is a closed 1-orbit of }H\}$ and choose $\ep>0$ sufficiently small such that
\begin{equation}
\label{eq:filt1}
A(H) \subset \bigcup_{p\in\Z} (G(p-1)+\ep,G(p)+\ep)
\end{equation}
and
\begin{equation}
\label{eq:filt2}
\# (A(H)\cap (G(p-1)+\ep,G(p)+\ep)) \leq 1
\end{equation}
for every $p\in\Z$ . Consider the filtration on the Floer complex $\CF_*(H)$ given by
\[
\F_p\CF_q(H) = \{x \in \CF_q(H);\, F(x)\leq G(p)+\ep\}.
\]
(As mentioned in the previous section, we actually have to take a generic non-degenerate time-dependent perturbation of $H$.) This filtration exhausts the complex in finitely many steps due to our assumptions on $H$ and the condition that $\lim_{p\to\infty} G(p)=\infty$. Consider the corresponding spectral sequence whose $E^0$-page is given by
\[
E^0_{p,q} =  \F_p\CF_{p+q}(H)/\F_{p-1}\CF_{p+q}(H).
\]
Using properties \eqref{eq:filt1} and \eqref{eq:filt2} of our filtration, we can show that the first page of this spectral sequence is given by
\[
E^1_{p,q} =
\begin{cases}
\bigoplus_{c} \HF^{loc}_{p+q}(B_{p,c},H) & 0<p\leq p_{max} \\
H_{q+n+1}(W,\partial W;\Lambda) & p=0\\
0 & p<0\text{ or }p>p_{max},
\end{cases}
\]
where the sum above runs over the connected components $B_{p,c}$ of the Morse-Bott components with $F$-action $p$ and $\HF^{loc}_{*}(B_{p,c},H)$ stands for the local Floer homology of $B_{p,c}$; see \cite{KvK,MR}. In general, this local Floer homology is isomorphic to the singular homology of $B_{p,c}$ twisted by a local system of coefficients, but it turns out that this system is trivial under our assumption on $\alpha$ as the following lemma shows.

\begin{lemma}[Lemma 8.9 in \cite{KvK} and Lemma 7.1 in \cite{MR}]
\label{lemma:lfh}
Suppose that $\alpha$ is Morse-Bott and, if $\alpha$ is degenerate, assume furthermore that the linearized Reeb flow is complex linear with respect to a unitary trivialization of the contact structure along every periodic orbit of $\alpha$. Then
\[
\HF^{loc}_{*}(B_{p,c},H) \cong H_{*+\mu(B_{p,c})}(B_{p,c};\Lambda)
\]
where $\mu(B_{p,c})$ is the Robbin-Salamon index of any closed orbit in $B_{p,c}$.
\end{lemma}

Thus, since our filtration is bounded and exhausting, we obtain convergent spectral sequences
\[
E^1_{p,q} \implies \HF_*(H)\quad\text{and}\quad E^{1,+}_{p,q} \implies \HF^+_*(H)
\]
whose first pages are given by
\[
E^1_{p,q} =
\begin{cases}
\bigoplus_{c} H_{p+q+\mu(B_{p,c})}(B_{p,c};\Lambda) & 0<p\leq p_{max} \\
H_{q+n+1}(W,\partial W;\Lambda) & p=0\\
0 & p<0\text{ or }p>p_{max},
\end{cases}
\]
and
\[
E^{1,+}_{p,q} =
\begin{cases}
\bigoplus_{c} H_{p+q+\mu(B_{p,c})}(B_{p,c};\Lambda) & 0<p\leq p_{max} \\
0 & p\leq 0\text{ or }p>p_{max}.
\end{cases}
\]
It turns out that analogous spectral sequences exist in the equivariant setup converging to $\EHF^{S^1}_*(H)$ and $\EHF^{+,S^1}_*(H)$.

Choosing a suitable sequence of Hamiltonians, we have that these spectral sequences pass to the limit, furnishing the following result.

\begin{theorem}[Corollary 7.2 in \cite{MR}]
Under our assumption on $\alpha$, there are convergent spectral sequences
\[
E^1_{p,q} \implies \ESH_*^{S^1}(W)\quad\text{and}\quad E^{1,+}_{p,q} \implies \ESH_*^{+,S^1}(W)
\]
whose first pages are given by
\[
E^1_{p,q} =
\begin{cases}
\bigoplus_{c} H^{S^1}_{p+q+\mu(B_{p,c})}(B_{p,c};\Lambda) & p>0 \\
H^{S^1}_{q+n+1}(W,\partial W;\Lambda) & p=0\\
0 & p<0,
\end{cases}
\]
and
\[
E^{1,+}_{p,q} =
\begin{cases}
\bigoplus_{c} H^{S^1}_{p+q+\mu(B_{p,c})}(B_{p,c};\Lambda) & p>0 \\
0 & p\leq 0,
\end{cases}
\]
where the $S^1$-action on $W$ is the trivial one so that $H^{S^1}_*(W,\partial W;\Lambda) \cong \bigoplus_{i+j=*} H_i(W,\partial W;\Lambda) \otimes H_j(\CP^\infty;\Lambda)$.
\end{theorem}

Clearly, all the constructions above respect the filtration given by the free homotopy classes in $W$. Using the spectral sequence $E^{1,+}_{p,q}$ established in the previous proposition, one can show that if $\alpha$ is non-degenerate then, given a set $\Gamma$ of free homotopy classes in $W$, $\HC_*^\Gamma(W)$ can be obtained as the homology of a chain complex $\CC^\Gamma_*(\alpha)$ with coefficients in the universal Novikov field generated by the good closed Reeb orbits of $\alpha$ with free homotopy class in $\Gamma$ graded by the corresponding Conley-Zehnder indices. As a matter of fact, this is proved when $W$ is a Liouville domain in \cite[Proposition 3.3]{GG} and \cite[Lemma 2.1]{GU} and one can check that the proof extends to our context since it is purely algebraic.

\section{Proof of Theorem \ref{thm:main}}
\label{sec:main}
We will prove the theorem when
\[
\dim \HC_{n}^\Gamma(W) < \dim \HC_{n+jK}^\Gamma(W)
\]
for some $K \in \N$ and every $j \in \N$, since the argument in the other case is analogous up to some sign changes. Throughout the proof we will use the following theorem taken from \cite[Theorem 4.1]{GGM2}. This result can also be derived from a refinement of the common index jump theorem proved in \cite{DLW}.

\begin{theorem}[\cite{GGM2}]
\label{thm:IRT}
Let $\Phi_1,\dots,\Phi_r$ be strongly non-degenerate paths in $\Sp(2n)$ starting at the identity with positive mean index. Given $N \in \N$ and $\ell_0 \in \N$ there exist two integer vectors $(d^+,k_1^+,\dots,k_r^+)$ and $(d^-,k_1^-,\dots,k_r^-)$ such that $d^\pm,k_1^\pm,\dots,k_r^\pm$ are all divisible by $N$ and
\begin{equation*}
\cz(\Phi_i^{k_i^\pm+\ell}) = d^\pm + \cz(\Phi_i^\ell),
\end{equation*}
\begin{equation*}
\cz(\Phi_i^{k_i^-}) - d^- = -(\cz(\Phi_i^{k_i^+}) - d^+),
\end{equation*}
and
\begin{equation*}
|\cz(\Phi_i^{k_i^\pm}) - d^\pm| \leq n
\end{equation*}
for every $i \in \{1,\dots,r\}$ and $\ell \in \Z$ in the range $1\leq |\ell| \leq \ell_0$.
\end{theorem}

Suppose that $\alpha$ has finitely many simple orbits $\ga_1,\dots,\ga_r$ with free homotopy class in $\Gamma$.
Here, \emph{simple} means that the closed orbit is not a multiple of a closed orbit with free homotopy class in $\Gamma$ (but it can be a multiple of an orbit with free homotopy class not in $\Gamma$). Note that $r\geq 1$ since $\HC_*^\Gamma(W)$ is not trivial. Arguing by contradiction, suppose that $\cz(\ga_i^k)=\cz(\ga_j^k)$ for every $i,j \in \{1,\dots,r\}$ and $k\in \N$. Consequently, $\HC_*^\Gamma(W)$ is the homology of a chain complex generated by the good closed orbits of $\alpha$ with trivial differential.

Let $\ga=\ga_1$. By hypothesis, we have that the mean index of $\ga$, denoted by $\Delta(\ga)$, is positive (otherwise, we would have that $\HC_k^\Gamma(W)=0$ for every $k\geq n$). Since $|k\Delta(\ga)-\cz(\ga^k)|<n$ for every $k \in \N$ we conclude that $\cz(\ga^k)>-n$ for every $k \in \N$. 

\begin{lemma}
\label{lemma:index jump}
There exists $\ell_0 \in \N$ such that $\cz(\ga^{k+\ell})>\cz(\ga^k)+2n$ for every $k \in \N$ and $\ell \geq \ell_0$.
\end{lemma}

\begin{proof}
Taking $\ell_0 \geq \frac{4n}{\Delta(\ga)}$ we have
\begin{align*}
\cz(\ga^{k+\ell}) & > (k+\ell)\Delta(\ga) - n \\
& > \cz(\ga^k) - 2n + \ell\Delta(\ga) \\
& \geq \cz(\ga^k) + 2n
\end{align*}
for every $\ell \geq \ell_0$ and $k \in \N$.
\end{proof}

Applying Theorem \ref{thm:IRT} to the linearized Reeb flow along $\ga$ with $N=2K$ and $\ell_0$ given by the previous lemma, we find even numbers $d$ and $k_\ga$ which are multiples of $K$ such that
\begin{equation}
\label{eq:IRT1}
\cz(\ga^{k_\ga \pm \ell})=d \pm \cz(\ga^\ell)
\end{equation}
for every $\ell$ such that $1\leq \ell \leq \ell_0$ and
\begin{equation}
\label{eq:IRT2}
d-n\leq \cz(\ga^{k_\ga})\leq d.
\end{equation}

\begin{lemma}
\label{lemma:iterate}
If $\cz(\ga^k) = d+n$ then $k=k_\ga+\ell$ for some $\ell$ such that $1\leq \ell < \ell_0$.
\end{lemma}

\begin{proof}
It follows from Lemma \ref{lemma:index jump} and \eqref{eq:IRT2} that
\begin{align*}
\cz(\ga^{k_\ga-\ell}) & < \cz(\ga^{k_\ga})-2n \\
& \leq d -  2n
\end{align*}
and
\begin{align*}
\cz(\ga^{k_\ga+\ell}) & > \cz(\ga^{k_\ga})+2n\\
& \geq d + n
\end{align*}
for every $\ell\geq \ell_0$. Since $\cz(\ga^k)>-n$ for every $k \in \N$, we deduce from \eqref{eq:IRT1} that
\begin{align*}
\cz(\ga^{k_\ga-\ell}) & = d - \cz(\ga^\ell) \\
& < d + n
\end{align*}
for every $1\leq \ell \leq \ell_0$. Finally, notice that, by \eqref{eq:IRT2},  $\cz(\ga^{k_\ga}) \neq d+n$.
\end{proof}

Now, the fact that $d$ is a multiple of $K$ and our hypothesis imply $\HC_{d+n}^\Gamma(W)\neq 0$. This implies that there exists $k \in \N$ such that $\cz(\ga_i^k)=d+n$ for every $i \in \{1,\dots,r\}$. By Lemma \ref{lemma:iterate} and our assumptions, $k=k_\ga+\ell$ for some $\ell$ such that $1\leq \ell < \ell_0$. Thus, it follows from \eqref{eq:IRT1} that $\cz(\ga_i^\ell)=n$  for every $i \in \{1,\dots,r\}$. This means that for every contribution to $\HC_{d+n}^\Gamma(W)$, necessarily given by $\ga_i^{k_\ga+\ell}$ with $1\leq \ell < \ell_0$ and $i \in \{1,\dots,r\}$, we get a contribution to $\HC_{n}^\Gamma(W)$ given by $\ga_i^\ell$ with $1\leq \ell < \ell_0$ and $i \in \{1,\dots,r\}$ (notice here that, since $k_\ga$ is even, $\ga_i^{k_\ga+\ell}$ is good if and only if $\ga_i^\ell$ is good). But this contradicts the assumption that $\dim \HC_{n}^\Gamma(W) < \dim \HC_{n+jK}^\Gamma(W)$ for every $j \in \N$.

\section{Proof of Corollary \ref{cor:quotient}}
\label{sec:quotient}

Let $\bM$ be a contact finite quotient of $M$ and $\tau: M \to \bM$ the corresponding quotient projection. Given a non-degenerate contact form $\balpha$ on $\bM$, consider its lift $\alpha=\tau^*\balpha$. Arguing by contradiction, suppose that $\balpha$ has only one simple closed orbit. By Theorem \ref{thm:main} and our assumptions, $\alpha$ has at least two simple closed orbits $\ga_1$ and $\ga_2$ such that
\begin{equation}
\label{eq:different indices}
\cz(\ga_1^k) \neq \cz(\ga_2^k)
\end{equation}
for some $k \in \N$, where the index is computed using a non-vanishing section $\sigma$ of $\Lambda^{n+1}_\C TW$. Since $\balpha$ has only one simple closed orbit, we have that $\ga_2=\psi\circ\ga_1$ for some deck transformation $\psi$.

Let $\xi$ be the contact structure on $M$ and note that $TW|_M=\xi \oplus \xi^\om$, where $\xi^\om$ is the trivial bundle given by the symplectic orthogonal of $\xi$. Choose a compatible almost complex structure $J$ on $W$ satisfying $J(\xi)=\xi$ and $J(\xi^\om)=\xi^\om$ so that $\Lambda^{n+1}_\C TW|_M=\Lambda^n_\C\xi \otimes \xi^\om$. Then $\sigma$ induces a non-vanishing section $\sigma_\xi$ of $\Lambda^n_\C\xi$ and it turns out that the grading of $\HC_*(W)$ induced by $\sigma$ (viewing $\HC_*(W)$ as the homology of a chain complex generated by the good closed orbits of a non-degenerate contact form) is given by the Conley-Zehnder index of the closed Reeb orbits defined using a trivialization of the contact structure (along each closed orbit) induced by $\sigma_\xi$. Given a closed orbit $\ga$, denote by $\cz(\ga,\sigma_\xi)$ the corresponding Conley-Zehnder index.

The deck transformation $\psi$ induces an isomorphism $\Psi: \Lambda^n_\C\xi \to \Lambda^n_\C\xi$, where we are identifying the top exterior power of $\xi$ with respect to the complex structures $J|_\xi$ and $\psi^*(J|_\xi)$. It turns out that
\[
\cz(\ga_1^j,\Psi^*\sigma_\xi)=\cz(\ga_2^j,\sigma_\xi)
\]
for every $j \in \N$. But, since $H^1(M,\R)=0$, we have that $\Psi^*\sigma_\xi$ and $\sigma_\xi$ are homotopic; see \cite[Lemma 4.3]{McL2}. Hence,
\[
\cz(\ga_1^j,\sigma_\xi)=\cz(\ga_2^j,\sigma_\xi)
\]
for every $j \in \N$, contradicting \eqref{eq:different indices}.

\begin{remark}
\label{rmk:quotient}
If the first Chern class of the contact structure $\bxi$ on $\bM$ vanishes, we do not need that $\Psi^*\sigma_\xi$ and $\sigma_\xi$ are homotopic to obtain a contradiction. Indeed, choose a section $\sigma_\bxi$ of $\Lambda^n_\C\bxi$ and consider its lift $\widetilde\sigma_\xi$ to $\Lambda^n_\C\xi$. Since $\widetilde\sigma_\xi$ is invariant by deck transformations, we have that
\[
\cz(\ga_1^j,\widetilde\sigma_\xi)=\cz(\ga_1^j,\Psi^*\widetilde\sigma_\xi)=\cz(\ga_2^j,\widetilde\sigma_\xi)
\]
for every $j \in \N$. However, note that $\widetilde\sigma_\xi$ is not necessarily induced by the section $\sigma$ of $\Lambda^{n+1}_\C TW$. In order to fix this, we claim that the equality $\cz(\ga_1^j,\widetilde\sigma_\xi)=\cz(\ga_2^j,\widetilde\sigma_\xi)$ implies that
\[
\cz(\ga_1^j,\sigma_\xi)=\cz(\ga_2^j,\sigma_\xi)
\]
for any section $\sigma_\xi$ of $\Lambda^n_\C\xi$. As a matter of fact, given a periodic orbit $\ga$ we have that the difference $\cz(\ga;\widetilde\sigma_\xi)-\cz(\ga;\sigma_\xi)$ depends only on the homology class $[\ga] \in H_1(M;\R)$; see \cite[Lemma 4.3]{McL2} and \cite[Remark 5.1]{MR}. Hence, it is enough to have that $\ga_1^j$ and $\ga_2^j$ represent the same homology class in $H_1(M;\R)$. But the deck transformations act trivially in $H_1(M;\R)$ and consequently $[\ga_1^j]=[\ga_2^j]$ for every $j \in \N$.

\end{remark}

\section{Proof of Corollary \ref{cor:displaceable}}
\label{sec:displaceable}

In this section we prove Corollary \ref{cor:displaceable} showing that the compact domain $W \subset X$ bounded by $M$ satisfies the hypothesis of Theorem \ref{thm:main}. More precisely, we have the following result.

\begin{proposition}
\label{prop:HC displaceable}
Let $(M^{2n+1},\xi)$ be a contact manifold admitting a displaceable exact contact embedding into an exact symplectic manifold $X$ which is convex at infinity and satisfies $c_1(TX)|_{\pi_2(X)}=0$. Denote by $W$ the compact region in $X$ bounded by $M$. We have that
\[
\dim \HC^0_{n+2k}(W)=\dim \HC^0_n(W)+1
\]
for every $k\in \N$.
\end{proposition}

\begin{proof}
Note that, since $X$ is exact, we can take the positive equivariant symplectic homology with rational coefficients. Since $M$ is displaceable, we have that the $S^1$-equivariant symplectic homology of $W$ vanishes (see \cite[Theorem 13.4]{Rit2} and \cite[Theorem 1.2]{BO17}). Using this and the Viterbo sequence (see e.g.~\cite{FSvK12,Ka}) we have that
\[
\HC^0_*(W) \cong \bigoplus_{i+j=*+n} H_i(W,\partial W;\Q) \otimes H_j(\CP^\infty;\Q).
\]
Therefore,
\[
\dim \HC^0_n(W) = \sum_{i=1}^{n} \dim H_{2i}(W,\partial W;\Q)
\]
and
\[
\dim \HC^0_{n+2k}(W) = \sum_{i=1}^{n+1} \dim H_{2i}(W,\partial W;\Q)
\]
for every $k\in \N$. Since $H_{2n+2}(W,\partial W;\Q) \cong \Q$, we conclude the desired result.
\end{proof}

\section{Proof of Corollary \ref{cor:cosphere}}
\label{sec:cosphere}

Let $\alpha$ be a non-degenerate contact form on $(S^* N,\xi_{can})$. We will split the proof according to the three hypotheses. In what follows, we will identify, without fear of ambiguity, $N$ with the subset of constant loops in $\LL N/S^1$ and $\LL N$. Moreover, we will take a non-vanishing section of $\Lambda^{n+1}_\C TD^*N$ induced from the choice of a volume form on $N$ whenever $N$ is orientable.

\subsubsection*{\bf (i) $N$ satisfies the first hypothesis}\hfill\\[-2ex]

Let $\NN \xrightarrow{\tau} N$ be the universal covering of $N$. Recall that $H_*(\LL \NN/S^1, \NN;\Q)$ is asymptotically unbounded if there exists a sequence of integers $k_i \to \infty$ such that $\dim H_{k_i}(\LL \NN/S^1, \NN;\Q) \to \infty$ as $i\to\infty$; otherwise, we call $H_*(\LL \NN/S^1, \NN;\Q)$ asymptotically bounded. It is not hard to see that the dimension of $H_*(\LL \NN/S^1,\NN;\Q)$ is asymptotically unbounded if and only if the dimension of $H_*(\LL \NN;\Q)$ is asymptotically unbounded.

Note that since $N$ is spin so is $\NN$: the corresponding second Stiefel-Whitney classes satisfy $w_2(\NN)=\tau^*w_2(N)=0$. By the isomorphism \eqref{eq:CL}, if $H_*(\LL \NN/S^1, \NN;\Q)$ is asymptotically unbounded so is $\HC_*(D^*\NN)$. Thus, in this case results from \cite{HM,McL} establish that every Reeb flow on $S^*\NN$ has infinitely many simple closed orbits. This implies that every Reeb flow on $S^*N$ has infinitely many simple closed orbits since $S^*N$ is a contact finite quotient of $S^*\NN$.

Assume then that $H_*(\LL \NN/S^1, \NN;\Q)$ is asymptotically bounded and let $m$ be the dimension of $N$. By Corollary \ref{cor:quotient}, using again the fact that $S^*N$ is a contact finite quotient of $S^*\NN$, it is enough to show that there exists $K \in \N$ such that
\[
\dim H_{m-1}(\LL \NN/S^1, \NN;\Q) < \dim H_{m-1+jK}(\LL \NN/S^1, \NN;\Q)\qquad \textrm{for every $j \in \N$.}
\]
(Note that $\pi_1(S^*\NN)$ is finite because $\NN$ is simply connected. Therefore, $H^1(S^*\NN,\R)=0$.) So, assume from now on that $N$ is simply connected. A result due to Vigu\'e-Poirrier and Sullivan \cite{VS} establishes that the dimension of $H_*(\LL N;\Q)$ is asymptotically bounded exactly when 
\[
H^*(N;\Q) \cong T_{d,h+1}(x)
\]
for some $d$ and $h$, where $T_{d,h+1}(x)$ is the truncated polynomial algebra with generator $x$ of degree $d$ and height $h+1$. In this case, Rademacher \cite{Rad2} computed $H_*(\LL N/S^1,N;\Q)$. First of all, $m=dh$. When $d$ is even, \cite[Theorem 2.4]{Rad2} tells us that the Poincar\'e series of $(\LL N/S^1,N)$ (for homology with rational coefficients) is given by
\begin{equation*}
P(\LL N/S^1,N;\Q)(t) = t^{d-1}\bigg(\frac{1}{1-t^2}+\frac{t^{d(h+1)-2}}{1-t^{d(h+1)-2}}\bigg)\frac{1-t^{dh}}{1-t^d}.
\end{equation*}
Therefore, we deduce from this that
\[
\dim H_{dh-1}(\LL N/S^1,N;\Q)=h\quad\text{and}\quad \dim H_{dh-1+j(d(h+1)-2)}(\LL N/S^1,N;\Q)=h+1
\]
for every $j \in \N$. When $d$ is odd, $h$ has to be equal to one and $N$ is rationally homotopic to the sphere $S^d$. The corresponding Poincar\'e series is given by (see  \cite[Remark 2.5]{Rad2})
\begin{equation*}
P(\LL N/S^1,N;\Q)(t) = t^{d-1}\bigg(\frac{1}{1-t^2}+\frac{t^{d-1}}{1-t^{d-1}}\bigg).
\end{equation*}
Consequently,
\[
\dim H_{d-1}(\LL N/S^1,N;\Q)=1\quad\text{and}\quad \dim H_{d-1+j(d-1)}(\LL N/S^1,N;\Q)=2
\]
for every $j \in \N$.

\subsubsection*{\bf (ii) $N$ satisfies the second hypothesis}\hfill\\[-2ex]

We will show that, under hypothesis (ii), $\HC^0_*(D^*N)\neq 0$ and $\HC^a_*(D^*N)\neq 0$ for any non-trivial homotopy class $a$. This readily implies the result in this case.

The non-triviality of $\HC^a_*(D^*N)$ easily follows from \eqref{eq:CL filtered} and the difficult part is to show that $\HC^0_*(D^*N)\neq 0$. In order to prove this, we first need the following well known result for which we give a proof for the convenience of the reader. For a shorter notation, we will suppress the superscript zero in $\LL^0 N$.

\begin{proposition}
\label{prop:relhom}
We have that $\pi_{j+1}(N) \cong \pi_{j}(\LL N,N)$ for every $j\geq 1$.
\end{proposition}

\begin{proof}
Let $\Om N$ be the based loop space of contractible loops. The canonical fibration $\Om N \hookrightarrow \LL N \xrightarrow{\text{ev}} N$ admits a section $N \xhookrightarrow{i} \LL N$ given by the inclusion of constant loops, and this implies that the associated homotopy long exact sequence reduces to the split short exact sequence
\[
\begin{tikzcd}
0 \ar[r] & \pi_*(\Om N) \ar[r,"\alpha"] & \pi_*(\LL N) \ar[r,"\beta"] & \pi_*(N) \ar[bend left=30]{l}{i} \ar[r] & 0
\end{tikzcd}
\]
so that $\beta \circ i = \Id$, where, for a shorter notation, we are using $i$ to denote both the inclusion $N \xhookrightarrow{i} \LL N$ and its map induced on the homotopy groups. In particular, $i$ is injective.

Since the sequence splits, we have that $\pi_j(\LL N) \cong \pi_j(\Om N) \oplus \pi_j(N)$ for every $j\geq 2$ and $\pi_1(\LL N) \cong \pi_1(\Om N) \rtimes \pi_1(N)$. However, using the well known isomorphism $\pi_*(\Om N) \cong \pi_{*+1}(N)$ and our assumptions that $\pi_1(N)\cong \Z$ and $\pi_2(N)=0$, we have that $\pi_j(\LL N) \cong \pi_j(\Om N) \oplus \pi_j(N) \cong \pi_{j+1}(N) \oplus \pi_j(N)$ for every $j\geq 1$ and consequently the fundamental group of $\LL N$ is abelian.

Thus, we have an inverse on the left of $\alpha$, $p: \pi_*(\LL N) \to \pi_*(\Om N)$, given by
\begin{equation}
\label{eq:map p}
p(b)=\alpha^{-1}(b-i(\beta(b)))
\end{equation}
and the commutative diagram
\begin{equation}
\label{eq:diag1}
\begin{tikzcd}
0 \ar[r] & \pi_*(\Om N) \ar[r,"\alpha"] & \pi_*(\LL N) \ar[r,"\beta"]\ar[d,"\psi"] & \pi_*(N) \ar[r] & 0 \\
& & \pi_*(\Om N) \oplus \pi_*(N) \ar[ul,"\tau_1"]\ar[ur,swap,"\tau_2"]
\end{tikzcd}
\end{equation}
where $\psi(b)=(p(b),\beta(b))$ is an isomorphism and $\tau_1$ and $\tau_2$ are the projections onto the first and second factors respectively.

On the other hand, the homotopy long exact sequence associated to the pair $(\LL N,N)$ furnishes the commutative diagram
\begin{equation}
\label{eq:diag2}
\begin{tikzcd}
[
column sep=small,row sep=small,
ar symbol/.style = {draw=none,"\textstyle#1" description,sloped},
isomorphic/.style = {ar symbol={\cong}},
]
\cdots \ar[r] & \pi_{k+1}(\LL N,N) \ar[r,"0"] & \pi_k(N) \ar[r,"i"]\ar[dr,"\vr"] & \pi_k(\LL N) \ar[r,"j"]\ar[d,"\psi"] & \pi_{k}(\LL N,N) \ar[r,"0"] & \cdots \\
& & & \pi_k(\Om N) \oplus \pi_k(N) \\
& & & \pi_{k+1}(N) \oplus \pi_k(N)  \ar[u,isomorphic]
\end{tikzcd}
\end{equation}
where $\vr:=\psi\circ i$ and we are using the isomorphism $\pi_k(\Om N) \cong \pi_{k+1}(N)$. Note that the map $\pi_{*}(\LL N,N) \to \pi_{*-1}(N)$ vanishes because $i$ is injective.

We claim that $\vr(c)=(0,c)$. In order to prove this, let us first prove that
\[
\tau_1 \circ \vr(c)=0.
\]
for every $c \in \pi_k(N)$. Indeed, we have that
\begin{align*}
\tau_1 \circ \vr(c) & = \tau_1 \circ \psi \circ i(c) \\
& = p(i(c)) \\
& = \alpha^{-1}(i(c)-i \circ \beta \circ i(c)) \\
& = \alpha^{-1}(i(c)-i(c)) \\
& = 0
\end{align*}
where the second equality follows from the fact that $\psi(i(c))=(p(i(c)),\beta(i(c)))$, the third equality holds because of \eqref{eq:map p} and the fourth equality follows from the relation $\beta \circ i = \Id$. Now, let us show that
\[
\tau_2 \circ \vr(c)=c.
\]
for every $c \in \pi_k(N)$. This is a consequence of the relation
\[
\tau_2 \circ \vr = \tau_2 \circ \psi \circ i = \beta \circ i = \Id
\]
where the second equality follows from the commutative diagram \eqref{eq:diag1}.

Finally, we conclude from the claim and the diagram \eqref{eq:diag2} that $j$ induces an isomorphism between $\pi_{k+1}(N)$ and $\pi_{k}(\LL N,N)$ for every $k\geq 1$, as desired.
\end{proof}

Now, we need the following result which is a rational relative Hurewicz theorem for the pair $(\LL N,N)$. Note that, by Proposition \ref{prop:relhom}, we have that $\pi_j(\LL N,N)$ is an abelian group for every $j\geq 1$.

\begin{proposition}
\label{prop:hurewicz}
Suppose that there exists $m\geq 1$ such that $\pi_j(\LL N,N) \otimes \Q=0$ for every $ j<m$ and $\pi_m(\LL N,N) \otimes \Q \neq 0$. Then $H_j(\LL N,N;\Q)=0$ for every $j<m$ and
\[
H_{m}(\LL N,N;\Q) \cong \pi'_{m}(\LL N,N) \otimes \Q
\]
where $\pi'_{m}(\LL N,N)$ is obtained from $\pi_m(\LL N, N)$ by factoring out the action of $\pi_1(N)$.
\end{proposition}

\begin{proof}
Recall that given a continuous map $f: X \to Y$ between topological spaces $X$ and $Y$ one can associate to it a fibration $F \to X' \to Y$ whose total space $X'$ is homotopically equivalent to $X$ and with fiber $F$ called the homotopy fiber. Consider such a fibration $F\to N' \to \LL N$ associated to the map $i: N \to \LL N$. 

\begin{lemma}
\label{lemma:homotopy F}
The homotopy fiber $F$ is homotopy equivalent to $\Om^2N$. In particular, we have that
\[
\pi_j(F) \cong \pi_{j+2}(N) \cong \pi_{j+1}(\LL N,N)
\]
for every $j \geq 0$. 
\end{lemma}

\begin{proof}
Consider the commutative diagram
\[
\begin{tikzcd}
[
column sep=small,row sep=small,
ar symbol/.style = {draw=none,"\textstyle#1" description,sloped},
isomorphic/.style = {ar symbol={\cong}},
]
N \ar[r,"i"]\ar[d,"\text{ev}"] & \LL N\ar[d,"\text{ev}"]  \\
N \ar[r,"\Id"] & N.
\end{tikzcd}
\]
This diagram has a fibre extension (up to homotopy) given by a $3 \times 3$ diagram such that all the rows and columns are fibration sequences (c.f. \cite[Section 3.2]{Nei})
\[
\begin{tikzcd}
[
column sep=small,row sep=small,
ar symbol/.style = {draw=none,"\textstyle#1" description,sloped},
isomorphic/.style = {ar symbol={\cong}},
]
\Om^2N \ar[r]\ar[d] & * \ar[r]\ar[d] & \Om N\ar[d] \\
F \ar[r]\ar[d] & N \ar[r]\ar[d] & \LL N\ar[d]  \\
* \ar[r] & N \ar[r] & N.
\end{tikzcd}
\]
where $*$ denotes a contractible space. From the first column we conclude that $F$ is homotopy equivalent to $\Om^2N$.

Consequently, $\pi_j(F) \cong \pi_{j}(\Om^2N) \cong \pi_{j+2}(N)$ for every $j \geq 0$. The isomorphism $\pi_j(F) \cong \pi_{j+1}(\LL N,N)$ follows from Proposition \ref{prop:relhom} (c.f. Remark \ref{rmk:relhom}).
\end{proof}

\begin{remark}
\label{rmk:relhom}
Since $F$ is the homotopy fiber of the inclusion $N \hookrightarrow \LL N$, it is well known that $\pi_*(F) \cong \pi_{*+1}(\LL N,N)$ (see, for instance, \cite[Page 155]{DK}). Consequently, one could also deduce Proposition \ref{prop:relhom} from the isomorphism $\pi_*(F) \cong \pi_{*+2}(N)$ established in the previous lemma.
\end{remark}

Therefore, we conclude from the previous lemma and our assumptions that $m\geq 2$, $F$ is connected and its fundamental group is abelian. Moreover, we have that $\pi_{m-1}(F) \otimes \Q \neq 0$ and, if $m>2$, $\pi_j(F) \otimes \Q=0$ for every $1\leq j \leq m-2$. The next lemma is a rational Hurewicz theorem for $F$.

\begin{lemma}
\label{lemma:hurewicz F}
We have that $H_j(F;\Q) \cong \pi_{j}(F) \otimes \Q$ for every $1 \leq j\leq m-1$.
\end{lemma}

\begin{proof}
Suppose initially that $G:=\pi_1(F)$ is torsion and let $\F \xrightarrow{\tau} F$ be the universal covering. By the rational Hurewicz theorem for simply connected spaces,
\begin{equation}
\label{eq:hurewicz FF}
H_q(\F;\Q) \cong \pi_q(\F) \otimes \Q
\end{equation}
for every $1\leq q\leq m-1$; see, for instance, \cite{KK}. Let $BG$ be the classifying space of $G$ and $\vr: F \to BG$ be the classifying map of $\tau$. The fibration associated to $\vr$ is
\[
\F \to F' \to BG
\]
with homotopy fiber $\F$ and total space $F'$ homotopically equivalent to $F$. Consider the corresponding Serre spectral sequence $E^r_{p,q}$ so that
\begin{equation}
\label{eq:2page}
E^2_{p,q} \cong H_p(BG;\underline{H_q}(\F;\Q)),
\end{equation}
where $\underline{H_q}(\F;\Q)$ is a local system with coefficients $H_q(\F;\Q)$. We claim that $H_p(BG;\Q)=0$ for all $p\geq 1$. Indeed, since $G$ is a finite abelian group, we can write
\[
G = \oplus_{i=1}^r \Z_{k_i}
\]
for positive integers $k_1,\dots,k_r$. Thus, we have the homotopy equivalence
\[
BG \cong \prod_{i=1}^r B\Z_{k_i}.
\]
But clearly $H_p(B\Z_{k_i};\Q)=0$ for every $p>0$ and therefore the claim follows from Kunneth theorem. 

Hence, we have from \eqref{eq:2page} that
\[
E^2_{p,q} = 0\quad\text{for all}\ p\geq 1\ \text{and}\ q\geq 0
\]
which implies that
\[
H_q(F;\Q) \cong E^\infty_{0,q} \cong E^2_{0,q} \cong H_0(BG;\underline{H_q}(\F;\Q)).
\]
But $H_0(BG;\underline{H_q}(\F;\Q))$ is isomorphic to $H_q(\F;\Q)$ factored out by the action of $\pi_1(BG)\cong G$ acting as the deck transformations of $\F \xrightarrow{\tau} F$.

By Lemma \ref{lemma:homotopy F}, $F$ is homotopy equivalent to $\Om^2N$ and therefore is an H-space. (Recall that a topological space $X$ is an H-space if there is a continuous multiplication map $\mu: X \times X \to X$ and an identity element $e \in X$ such that the maps $\mu(\cdot,e)$ and $\mu(e,\cdot)$ are homotopic to the identity through maps $(X,e) \to (X,e)$.) Therefore, $G$ acts trivially on $H_*(\F;\Q)$. As a matter of fact, note that $\F$ is also an H-space and $\tau$ is an H-map (i.e. if $\mu$ and $\tilde\mu$ are the multiplications on $F$ and $\F$ respectively then the maps $\tau\circ\tilde\mu$ and $\mu\circ (\tau\times\tau)$ are homotopic). Let $G_e \subset \F$ be the fiber of $\tau$ over the identity element in $F$ so that there is a natural identification $G\cong G_e$. Then the deck transformation of $\tau$ corresponding to $g\in G$ is given by $\tilde\mu(\tilde g,\cdot)$ where $\tilde g \in G_e$ corresponds to $g$ via the identification $G\cong G_e$. Since $\F$ is path connected, such map must be homotopic to the identity and therefore $G$ acts trivially on $H_*(\F;\Q)$.

Thus, by \eqref{eq:hurewicz FF},
\[
H_0(BG;\underline{H_q}(\F;\Q)) \cong H_q(\F;\Q)\cong \pi_q(\F) \otimes \Q
\]
for every $1\leq q \leq m-1$. But $\pi_j(\F) \cong \pi_j(F)$ for every $j\geq 2$ and $\pi_1(\F) \otimes \Q =\pi_1(F) \otimes \Q = 0$ and therefore
\[
\pi_*(\F) \otimes \Q \cong \pi_*(F) \otimes \Q
\]
implying that
\[
H_q(F;\Q) \cong \pi_q(F)\otimes \Q
\]
for every $1\leq q\leq m-1$, proving the lemma when $\pi_1(F)$ is torsion.

Finally, when $\pi_1(F)$ is not torsion, we have by Hurewicz theorem and Lemma \ref{lemma:homotopy F} that $H_1(F;\Z) \cong \pi_1(F) \cong \pi_2(\LL N,N)$ (note that $\pi_1(F)$ is abelian) and hence $m=2$ and
\[
H_1(F;\Q) \cong \pi_1(F) \otimes \Q,
\]
as desired.
\end{proof}

Now, let $E^r_{p,q}$ be the Serre spectral sequence associated to $F\to N' \to \LL N$ so that
\[
E^2_{p,q} \cong H_p(\LL N;\underline{H_q}(F;\Q)),
\]
where $\underline{H_q}(F;\Q)$ is a local system with coefficients in $H_q(F;\Q)$ such that the action $\pi_1(\LL N) \times  H_q(F;\Q) \to H_q(F;\Q)$ fits in the commutative diagram
\begin{equation}
\label{diag:action}
\begin{tikzcd}
[
column sep=small,row sep=small,
ar symbol/.style = {draw=none,"\textstyle#1" description,sloped},
isomorphic/.style = {ar symbol={\cong}},
]
\pi_1(\LL N) \times  H_q(F;\Q) \ar[r] & H_q(F;\Q)  \\
\pi_1(N) \times  \pi_{q+1}(\LL N,N) \otimes \Q \ar[u,isomorphic] \ar[r]  & \pi_{q+1}(\LL N,N) \otimes \Q \ar[u,isomorphic]
\end{tikzcd}
\end{equation}
for all $1\leq q\leq m-1$, where in the second line we have the usual action of $\pi_1(N)$ on $\pi_{q+1}(\LL N,N) \otimes \Q$ and the vertical isomorphisms follow from Lemma \ref{lemma:hurewicz F} and the isomorphism $\pi_q(F) \cong \pi_{q+1}(\LL N,N)$.

Since $F$ is connected, $\underline{H_0}(F;\Q)$ is a trivial system and hence
\begin{equation}
\label{eq.Ep0}
E^2_{p,0} \cong H_p(\LL N;H_0(F;\Q)) \cong H_p(\LL N;\Q)
\end{equation}
for every $p\geq 0$. By Lemmas \ref{lemma:homotopy F} and \ref{lemma:hurewicz F},
\[
H_{m-1}(F;\Q) \cong \pi_{m-1}(F) \otimes \Q \cong \pi_{m}(\LL N,N) \otimes \Q \neq 0
\]
and, if $m>2$,
\[
H_q(F;\Q) = \pi_q(F) \otimes \Q = \pi_{q+1}(\LL N,N) \otimes \Q = 0\ \text{for every}\ 1\leq q \leq m-2.
\]
Using this we deduce that
\begin{equation}
\label{eq:E0l-1}
E^2_{0,m-1} \cong \pi'_{m}(\LL N,N) \otimes \Q
\end{equation}
and, if $m>2$,
\begin{equation}
\label{eq:Epq}
E^2_{p,q} = 0
\end{equation}
for every $p\geq 0$ and $1 \leq q\leq m-2$, where in \eqref{eq:E0l-1} we are using the commutative diagram \eqref{diag:action} (c.f. \cite[Theorem VI.3.2]{Whi}).

The spectral sequence converges to
\begin{equation}
\label{eq:E^infty}
E^\infty_{p,q} \cong
\begin{cases}
0 & \text{ for every } p\text{ and }q>0 \\
H_p(N;\Q) &  \text{ for every } p\text{ and }q=0.
\end{cases}
\end{equation}
Indeed, by the construction of the spectral sequence, $E^\infty_{p,0}$ is a quotient of $H_p(N;\Q)$ and therefore we have a map $\alpha_p: H_p(N;\Q) \to E^\infty_{p,0}$ (c.f. \cite[Section 7, Chapter XIII]{Whi}). The composition
\[
H_p(N;\Q) \xrightarrow{\alpha_p} E^\infty_{p,0} \hookrightarrow H_p(\LL N;\Q)
\]
is called the edge homomorphism and it turns out that it is the map induced by the projection of the fibration (c.f. \cite[Theorem XIII.7.2]{Whi}) which in our case coincides with the map induced by the inclusion $N \xhookrightarrow{i} \LL N$. Since this map is injective, we conclude that nothing can survive in $E^\infty_{p,q}$  above the line zero and consequently we obtain the isomorphism \eqref{eq:E^infty}.

From \eqref{eq.Ep0}, \eqref{eq:E0l-1}, \eqref{eq:Epq} and \eqref{eq:E^infty} we easily conclude that, for every $0\leq j \leq m$, the differential $d^{j}_{j,0}$ from $E^{j}_{j,0} \cong E^2_{j,0} \cong H_{j}(\LL N;\Q)$ to $E^{j}_{0,j-1} \cong E^2_{0,j-1} \cong \pi'_{j}(\LL N,N) \otimes \Q$ induces an isomorphism
\begin{equation}
\label{eq:seq iso}
\pi'_{j}(\LL N,N) \otimes \Q \cong H_{j}(\LL N;\Q)/i_*(H_{j}(N;\Q)).
\end{equation}
The evaluation map $ev: \LL N \to N$ and $i: N \to \LL N$ satisfy $ev \circ i = \Id$ and consequently $i_*: H_*(N;\Q) \to H_*(\LL N;\Q)$ is injective. Therefore, the homology long exact sequence associated to the pair $(\LL N,N)$ readily implies that
\[
H_{*}(\LL N,N;\Q) \cong H_{*}(\LL N;\Q)/i_*(H_{*}(N;\Q))
\]
and the proposition follows from \eqref{eq:seq iso}.
\end{proof}

Now, we claim that there exists $j>1$ such that $\pi_j(N) \otimes \Q \neq 0$. Arguing by contradiction, suppose that $\pi_j(N) \otimes \Q = 0$ for every $j>1$. Recall that a topological space $X$ is nilpotent if $\pi_1(X)$ is nilpotent (e.g. if $\pi_1(X)$ is abelian) and the action of $\pi_1(X)$ on $\pi_j(X)$ is nilpotent (e.g. if $\pi_j(X)$ is finite) for every $j>1$; see \cite{HMR,MP} for precise definitions. Then, by our assumptions, $N$ is a nilpotent space. Let $m>1$ be the dimension of $N$ and $f: S^1 \to N$ be a closed curve whose homotopy class is a generator of $\pi_1(N) \cong \Z$. Then $f$ induces an isomorphism $f_*: \pi_*(S^1)\otimes \Q \to \pi_*(N)\otimes \Q$. But by the Whitehead-Serre theorem for nilpotent spaces (c.f. \cite[Theorems 3.3.8 and 5.3.2]{MP}), this implies that $f$ also induces an isomorphism $H_*(S^1;\Q) \to H_*(N;\Q)$, contradicting the fact that $H_m(N;\Q) \cong \Q$.

Now, let $k$ be the smallest integer bigger than one such that $\pi_k(N) \otimes \Q \neq 0$. It follows from Proposition \ref{prop:relhom} that $\pi_{j}(\LL N,N) \otimes \Q = 0$ for every $j<k-1$ and $\pi_{k-1}(\LL N,N) \otimes \Q \neq 0$. The action $\pi_1(N) \times \pi_{k-1}(\LL N,N) \to \pi_{k-1}(\LL N,N)$ fits in the commutative diagram
\begin{equation}
\begin{tikzcd}
[
column sep=small,row sep=small,
ar symbol/.style = {draw=none,"\textstyle#1" description,sloped},
isomorphic/.style = {ar symbol={\cong}},
]
\pi_1(N) \times \pi_{k-1}(\LL N,N) \ar[r] & \pi_{k-1}(\LL N,N)  \\
\pi_1(N) \times  \pi_{k}(N) \ar[u,isomorphic] \ar[r]  & \pi_{k}(N) \ar[u,isomorphic],
\end{tikzcd}
\end{equation}
where in the second line we have the usual action of $\pi_1(N)$ on $\pi_{k}(N)$ and the vertical isomorphisms follow from Proposition \ref{prop:relhom}.

Thus, since $N$ is $k$-simple, we have that $\pi_{k-1}(\LL N,N) \cong \pi'_{k-1}(\LL N,N)$ and hence, by Proposition \ref{prop:hurewicz},
\begin{equation}
\label{eq:non-trivial hom}
H_{k-1}(\LL N,N;\Q) \neq 0.
\end{equation}

\begin{remark}
\label{rmk:simple}
Notice that to obtain this inequality we actually only need the following weaker assumption: let $a$ be a generator of $\pi_1(N)$ and denote by $A$ the linear map corresponding to the action of $a$ on $\pi_k(N) \otimes \Q \cong \pi_{k-1}(\LL N,N) \otimes \Q$. Since $\pi'_{k-1}(\LL N,N) \otimes \Q$ is the quotient of $\pi_{k-1}(\LL N,N) \otimes \Q$ by the subgroup $\{A^ng-g;\ n\in \Z,\ g \in \pi_{k-1}(\LL N,N) \otimes \Q\}$, it is clearly enough that $\ker(A-\Id)\neq 0$.
\end{remark}

We claim that this implies that $H_*(\LL N/S^1,N;\Q)$ is not trivial, that is, $H_j(\LL N/S^1,N;\Q)\neq 0$ for some degree $j$. As a matter of fact, this is a consequence of the following proposition.

\begin{proposition}
If the total homology $H_*(\LL N/S^1,N;\Q)$ is trivial then so is $H_*(\LL N,N;\Q)$.
\end{proposition}

\begin{proof}
Applying the relative Gysin sequence to the $S^1$-bundle $\LL N \times ES^1 \to \LL N \times_{S^1} ES^1$ we obtain the exact triangle
\[
\begin{tikzcd}
H_*(\LL N \times ES^1,N \times ES^1;\Q) \arrow{r} \arrow[leftarrow]{dr}[swap]{[+1]} & H_*(\LL N \times_{S^1} ES^1,N \times_{S^1} ES^1;\Q) \arrow{d} \\
     & H_{*-2}(\LL N \times_{S^1} ES^1,N \times_{S^1} ES^1;\Q)
\end{tikzcd}
\]
which yields
\[
\begin{tikzcd}
H_*(\LL N,N;\Q) \arrow{r} \arrow[leftarrow]{dr}[swap]{[+1]} & H_*^{S^1}(\LL N,N;\Q) \arrow{d} \\
     & H_{*-2}^{S^1}(\LL N,N;\Q).
\end{tikzcd}
\]
Since $H_*(\LL N/S^1,N;\Q) \cong H_*^{S^1}(\LL N,N;\Q)$ we immediately conclude the result.
\end{proof}

Thus, we conclude from \eqref{eq:CL filtered}, \eqref{eq:non-trivial hom} and the previous proposition that
\[
\HC_{*}^0(D^*N) \neq 0
\]
and consequently $\alpha$ carries a contractible periodic orbit $\ga_1$. Indeed, if $\alpha$ has no contractible closed orbit then we would have that $\HC_{*}^0(D^*N) =0$, furnishing a contradiction. Note that we do not need the non-degeneracy of $\alpha$ here.

\begin{remark}
\label{rmk:LF}
The proof above shows that if $\pi_k(N) \otimes \Q \neq 0$ for some $k>1$ then $\alpha$ carries a contractible periodic orbit assuming only that $N$ is oriented spin, $\pi_1(N)$ is abelian, $\pi_2(N)=0$ and $N$ is $k$-simple. The hypothesis that $N$ is $k$-simple can be relaxed by the assumption that the subgroup $\{ag-g;\ a\in \pi_1(N),\ g \in \pi_{k}(\N) \otimes \Q\}$ is strictly contained in $\pi_{k}(N) \otimes \Q$. When $\pi_1(N)\cong \Z$ this is equivalent to ask that $\ker(A-\Id)\neq 0$, where $A$ is 
the linear map corresponding to the action of a generator of $\pi_1(N)$ on $\pi_k(N) \otimes \Q$; see Remark \ref{rmk:simple}.
\end{remark}

To find a second closed orbit, take a non-trivial homotopy class $a \in \pi_1(N) \cong \pi_1(D^*N)$. Clearly, we have that
\[
H_*(\LL^a N/S^1;\Q) \neq 0.
\]
This, together with \eqref{eq:CL filtered}, implies that $\alpha$ has a periodic orbit $\ga_2$ with homotopy class $a$. Since $\pi_1(D^*N) \cong \pi_1(N) \cong \Z$, $a$ is not nilpotent (i.e. there is no $k\in \N$ such that $a^k=0$) and consequently we have that $\ga_1$ and $\ga_2$ are geometrically distinct.

\subsubsection*{\bf (iii) $N$ satisfies the third hypothesis}\hfill\\[-2ex]

By hypothesis, given a non-trivial free homotopy class $a$ in $N$ there exists another non-trivial free homotopy class $b$ such that $b\neq a^n$ for every $n\in \Z$. Thus, we deduce from \eqref{eq:CL filtered} that (notice that we can naturally identify the free homotopy classes of $N$ and $D^*N$)
\[
\HC_{*}^a(D^*N) \cong H_*(\LL^a M/S^1;\Q) \neq 0
\]
and
\[
\HC_{*}^b(D^*N) \cong H_*(\LL^b M/S^1;\Q) \neq 0.
\]
From this we get two closed orbits $\ga_1$ and $\ga_2$ with free homotopy classes $a$ and $b$ respectively. From our choice of $a$ and $b$ we conclude that these orbits must be geometrically distinct.

\section{Proof of Corollary \ref{cor:geodesics}}
\label{sec:geodesics}

As in the proof of Corollary \ref{cor:cosphere}, we will consider each hypothesis separately. Throughout this section, we will tacitly identify a closed geodesic $\ga$ with the corresponding closed orbit of the geodesic flow restricted to the unit sphere bundle. As in the previous section, if $N$ is orientable we will take a non-vanishing section of $\Lambda^{n+1}_\C TD^*N$ induced from a volume form on $N$ so that the Conley-Zehnder index of a non-degenerate closed geodesic coincides with its Morse index.

\subsubsection*{\bf (i) $N$ satisfies the first hypothesis}\hfill\\[-2ex]

Arguing as in the proof of Corollary \ref{cor:cosphere}, it is enough to show that if $N$ is simply connected then every bumpy Finsler metric $F$ has either infinitely many geometrically distinct closed geodesics or at least two geometrically distinct closed geodesics $\ga_1$ and $\ga_2$ such that $\cz(\ga_1^k)\neq \cz(\ga_2^k)$ for some $k \in \N$.

If $H_*(\LL N/S^1, N;\Q)$ is asymptotically unbounded then the Gromoll-Meyer theorem \cite{GM} implies that $F$ has infinitely many geometrically distinct closed geodesics. (Note here that the proof the Gromoll-Meyer theorem works for Finsler metrics.) When $H_*(\LL N/S^1, N;\Q)$ is asymptotically bounded, then, as explained in the proof of Corollary \ref{cor:cosphere}, there exists $K \in \N$ such that
\[
\dim H_{m-1}(\LL N/S^1, N;\Q) < \dim H_{m-1+jK}(\LL N/S^1, N;\Q)\qquad \textrm{for every $j \in \N$}
\]
where $m$ is the dimension of $N$.

Arguing by contradiction, suppose that $F$ has only one prime closed geodesic. As explained in the introduction, the proof of Theorem \ref{thm:main} uses only the fact that given a non-degenerate contact form $\alpha$ on $M$ then $\HC_*^\Gamma(W)$ is the homology of a chain complex generated by the good periodic orbits of $\alpha$ with free homotopy class in $\Gamma$ graded by the Conley-Zehnder index. Thus, it is enough to show that $H_*(\LL N/S^1, N;\Q)$ is the homology of a chain complex generated by the good periodic orbits of the geodesic flow of $F$. (Here, the geodesic flow of $F$ means the geodesic flow of $F$ restricted to the unit sphere bundle.)

In order to do this, consider the energy functional $E: \LL N \to \R$ induced by $F$. Let $\ga_1,\dots,\ga_r$ be the simple closed orbits of the geodesic flow of $F$. By our assumption, $r=1$ if $F$ is not reversible or $r=2$ and $\ga_2=-\ga_1$ if $F$ is reversible, where $-\ga_1$ denotes the corresponding reversed geodesic. In what follows, we will use well known facts about closed geodesics; see e.g. \cite{BL,Kli,Rad}. Given a closed geodesic $\ga$, define $\LL(\ga)=E^{-1}([0,E(\ga)))$. The (equivariant) Morse type numbers are defined as
\[
m_k = \sum_{\substack{i \in \{1,\dots,r\} \\ j \in \N}} \dim H_k((\LL(\ga_i^j)\cup S^1\cdot\ga_i^j)/S^1,\LL(\ga_i^j)/S^1;\Q)
\]
where $S^1\cdot\ga_i^j$ is the orbit of $\ga_i^j$ in $\LL N$ induced by the (obvious) $S^1$-action. Since
\[
H_k((\LL(\ga)\cup S^1\cdot\ga)/S^1,\LL(\ga)/S^1;\Q) \cong
\begin{cases}
\Q\quad \text{if}\ k=\cz(\ga)\ \text{and}\ \ga\ \text{is good} \\
0\quad\text{otherwise}
\end{cases}
\]
for every closed geodesic $\ga$, we have that $m_k$ is the number of good orbits with index $k$. Consider the Betti numbers
\[
b_k = \dim H_k(\LL N/S^1,N;\Q).
\]
We have the Morse inequalities
\[
m_k - m_{k-1} + m_{k-2} - \cdots + (-1)^km_0 \geq b_k - b_{k-1} + b_{k-2} - \cdots + (-1)^kb_0
\]
for every $k \geq 0$. Since $F$ has only one prime closed geodesic, the index of every good closed orbit has the same parity $p$. Hence,
\[
m_k = 0\quad \text{for every}\ k \neq p\ \text{(mod 2)}.
\]
This fact, together with the Morse inequalities, implies that
\[
b_k = m_k\quad \text{for every}\ k.
\]
Thus, we conclude that $H_*(\LL N/S^1, N;\Q)$ is the homology of the chain complex with rational coefficients and trivial differential generated by the good periodic orbits of the geodesic flow of $F$ graded by the Conley-Zehnder index, as desired.

\subsubsection*{\bf (ii) $N$ satisfies the second hypothesis}\hfill\\[-2ex]

Since $N$ is not diffeomorphic to the the circle and $\pi_1(N)\cong \Z$, there is $j\geq2$ such that $\pi_j(N)\neq0$.  Let $k$ be the smallest integer bigger than one such that $\pi_k(N)\neq 0$. Then, by Proposition \ref{prop:relhom}, we have that $\pi_{k-1}(\LL^0 N,N) \cong \pi_k(N) \neq 0$. Then, applying classical minimax arguments to a non-trivial homotopy class in $\pi_{k-1}(\LL^0 N,N)$ (see e.g.~\cite{Kli}) we obtain a contractible closed geodesic $\ga_1$.

Now, taking a non-trivial homotopy class $a \in \pi_1(N)$ one can take a  minimum of $E|_{\LL^a N}$ which furnishes a non-contractible closed geodesic $\ga_2$. Since $\pi_1(N) \cong \Z$ we conclude that $\ga_1$ and $\ga_2$ must be geometrically distinct.

\subsubsection*{\bf (iii) $N$ satisfies the third hypothesis}\hfill\\[-2ex]

By hypothesis, given a non-trivial free homotopy class $a$ in $N$ there exists another non-trivial free homotopy class $b$ such that $b\neq a^n$ for every $n\in \Z$. Minimizing the energy functional on $\LL^a N$ and $\LL^b N$ we obtain two closed geodesics $\ga_1$ and $\ga_2$ with free homotopy classes $a$ and $b$ which have to be geometrically distinct due to our choice of $a$ and $b$.

\section{Proof of Corollary \ref{cor:toric}}
\label{sec:toric}

Let $(M^{2n+1}, \xi)$ be a Gorenstein toric contact manifold (i.e. a toric contact manifold with vanishing first Chern class) determined by a toric diagram $D\subset\R^n$ (see~\cite{AM3}). Let $F$ denote one of the facets of $D$, necessarily a simplex, and assume without loss of generality that its vertices are  $v_1, \ldots, v_n \in \Z^n$. Let $\eta\in\Z^n$ be such that
\[
\{ \nu_1 = (v_1,1), \ldots, \nu_n = (v_n,1), (\eta, 1)\} \ \text{is a $\Z$-basis of $\Z^{n+1}$.}
\]
Let $R$ denote a rational toric Reeb vector, which can be uniquely written as
\[
R = \sum_{j=1}^n b_j \nu_j + b (\eta,1)\,,\ \text{with}\ b_1,\ldots, b_n, b \in \Q\ \text{and}\ 
b = 1 - \sum_{j=1}^n b_j \ne 0\,.
\]
Write
\[
\frac{b_j}{|b|} = \frac{p_j}{q_j} \ \text{with}\ p_j \in \Z\,,\ q_j \in\N\ \text{and}\ \gcd (p_j,q_j) =1\,.
\]
Let
\begin{equation} \label{def:q}
q := \lcm (q_1,\ldots,q_n)
\end{equation}
so that $q b_j / |b| \in \Z$ for all $j=1,\ldots,n$. We then have that
\begin{equation} \label{def:N}
b + \sum_{j=1}^n b_j = 1 \ \Rightarrow\  N:= \frac{q}{|b|} \in\N\,.
\end{equation}

We will consider irrational perturbations $R_\delta$ of $R$ that can be uniquely written as
\[
R_\delta = \sum_{j=1}^n b_{\delta,j} \nu_j + b_\delta (\eta,1)\,,\ \text{with}\ b_{\delta,1},\ldots, b_{\delta,n}, b_\delta \in \R\,,\ 
b_\delta = 1 - \sum_{j=1}^n b_{\delta,j} \ne 0\ \text{and}
\]
\[
\frac{b_{\delta,j}}{|b_\delta |} = \frac{b_j}{|b|} + \delta_j \ \text{for some}\ \delta_j \in\R\setminus \{0\}\,,\ j=1,\ldots,n\,.
\]
We will say that $R_\delta$ is $\epsilon$-small for some $\epsilon >0$ whenever $0 < |\delta_j| <\epsilon$ for all $j=1,\ldots,n$.
We will denote by $\gamma$ the simple closed $R$-orbit associated with $F$ and by $\gamma_\delta$ the simple closed 
$R_\delta$-orbit associated with $F$.

The Conley-Zehnder index of $\gamma_\delta^{k}$, for any  $k\in\N$, is given by (see~\cite[Section 5]{AM1} 
and~\cite[Section 3]{AMM})
\begin{equation} \label{eq:CZ1}
\cz(\gamma_\delta^{k}) = 2  \sum_{j=1}^n \left\lfloor k \left( \frac{b_j}{|b|} + \delta_j\right) \right\rfloor 
 + 2k \frac{b}{|b|} + n\,.
\end{equation}
This implies that the Conley-Zehnder index of $\gamma_\delta^{k_1 + k_2 q}$, for any  $k_1,k_2\in\N$ and $q\in\N$
defined by~(\ref{def:q}), is given by
\begin{equation} \label{eq:CZ2}
\begin{split}
\cz(\gamma_\delta^{k_1 + k_2 q})  & =  2  \sum_{j=1}^n \left(\left\lfloor k_1 \frac{b_j}{|b|} + (k_1 + k_2 q)\delta_j \right\rfloor
+ k_2 q  \frac{b_j}{|b|}\right) + 2 (k_1 + k_2 q) \frac{b}{|b|} + n \\
& = 2  \sum_{j=1}^n \left\lfloor k_1 \frac{b_j}{|b|} + (k_1 + k_2 q)\delta_j \right\rfloor + 2 k_2 \frac{q}{|b|} + 2 k_1 \frac{b}{|b|} + n.
\end{split}
\end{equation}

\begin{prop} \label{prop:CZ}
For any $L\in\N$ there exists $\epsilon > 0$ such that, for any $\epsilon$-small $R_\delta$ and any $k_1,k_2\in\N$ with $k_1,k_2\leq L$,
we have
\[
	\cz(\gamma_\delta^{k_1 + k_2 q})  = \cz(\gamma_\delta^{k_1}) + 2 k_2 N \,,
\]
where $q,N\in\N$ are defined by~(\ref{def:q}) and~(\ref{def:N}). Moreover, denoting by $s$ the number of negative
$\delta_j$'s, $j=1,\ldots,n$, of the $\epsilon$-small $R_\delta$, we also have that
\[
	\cz(\gamma_\delta^{k_2 q})  = n - 2s + 2 k_2 N \,.
\]
\end{prop}
\begin{proof}
Using~(\ref{eq:CZ1}) and~(\ref{eq:CZ2}), it suffices to choose $\epsilon > 0$ small enough so that
\[
\left\lfloor k_1 \frac{b_j}{|b|} + k_1 \delta \right\rfloor = 
\left\lfloor k_1 \frac{b_j}{|b|} + k_1 \delta + k_2 q \delta \right\rfloor \]
for all $k_1,k_2\leq L$, $j=1,\ldots,n$, and any $0 < |\delta| <\epsilon$. 
\end{proof}

\begin{thm} \label{thm:HC toric}
Let $(M^{2n+1}, \xi)$ be a good toric contact manifold admitting a strong symplectic filling $W$ such that $c_1(TW)=0$. 
There exists $K \in \N$ such that
\[
\dim \HC^0_{n-2s}(W) < \dim \HC^0_{n-2s + d K}(W) \quad\text{for all}\quad s=0,1,\ldots,n\,,\ d\in\N\,.
\]
\end{thm}
\begin{proof}
Let $D\subset\R^n$ be the toric diagram of $(M^{2n+1}, \xi)$. Pick a rational toric Reeb vector $R$, denote 
by $\gamma_1,\ldots,\gamma_m$ the simple closed $R$-orbits, associated with the $m$ facets of $D$, and
by $N_1, \ldots, N_m$ the corresponding natural numbers given by~(\ref{def:N}). Define
\[
K := 2\lcm (N_1, \ldots, N_m,p)\,,
\]
where $p:=|\pi_1 (M)|$ is the order of the fundamental of group pf $M$.

It follows from the first part of Proposition~\ref{prop:CZ}, by considering arbitrarily large $L\in\N$ and 
$k_1,k_2$ multiples of $p$, that
\[
\dim \HC^0_{d_1}(W) \leq \dim \HC^0_{d_1 + d_2 K}(W) \quad\text{for all}\quad d_1,d_2\in\N\,.
\]
On the other hand, one can choose the rational toric Reeb vector $R$ so that one of the above simple closed orbits, 
say $\gamma_1$, has an arbitrarily high mean index, in particular greater than $2n$. This implies that
$\mu_{\rm CZ} (\gamma_{1,\delta}^{k}) > n$ for any $\epsilon$-small $R_\delta$, with $\epsilon$ small enough, and
any $k\in\N$. Hence, there is no contribution of any iterate of $\gamma_{1,\delta}$ to $\HC^0_{n-2s} (W)$, for any 
$s=0,1,\ldots,n$, but it follows from the second part of Proposition~\ref{prop:CZ} that, given $s=0,1,\ldots,n$, 
we can choose such an $R_\delta$ so that contractible iterates of $\gamma_{1,\delta}$ do contribute to $\HC^0_{n-2s + dK}(W)$
for all $d\in\N$ up to an arbitrarily large $L\in\N$ . Hence, we can indeed conclude that
\[
\dim \HC^0_{n-2s}(W) < \dim \HC^0_{n-2s + d K}(W) \quad\text{for all}\quad s=0,1,\ldots,n\,,\ d\in\N\,.
\]
\end{proof}

Corollary \ref{cor:toric} readily follows from this theorem (with $s=0$) and Corollary \ref{cor:quotient}.

\section{Proof of Corollary \ref{cor:prequantization}}
\label{sec:prequantization}

Corollary \ref{cor:prequantization} is a consequence of Theorem \ref{thm:main}, Corollary \ref{cor:quotient} and the following proposition.

\begin{proposition}
Let $(M^{2n+1},\xi)$ be a prequantization circle bundle of a closed integral symplectic manifold $(B,\om)$ such that $\om|_{\pi_2(B)}\neq 0$, $c_1(TB)|_{\pi_2(B)}\neq 0$ and, furthermore, $H_{k}(B;\Q)=0$ for every odd $k$ or $c_B>n$. Suppose that $M$ admits a strong symplectic filling $W$ such that $c_1(TW)=0$. Then
\[
\dim \HC_{n}^0(W) < \dim \HC_{n+jc_B}^0(W)
\]
or
\[
\dim \HC_{-n}^0(W) < \dim \HC_{-n-jc_B}^0(W)
\]
for every $j \in \N$.
\end{proposition}

\begin{proof}
Since $c_1(TW)=0$, we have that $c_1(\xi)=0$ and this implies that $c_1(TB)=\rho[\om]$ for some $\rho \in \R$. The constant $\rho$ must be different from zero by our assumption on $c_1(TB)$. Suppose initially that $\rho>0$. Arguing as in the proof of \cite[Theorem 3.1(a)]{GGM2} we conclude that
\begin{equation}
\label{eq:CH+}
\dim \HC^0_k(W) = \sum_{m\in\N} \dim H_{k-2mc_B+n} (B; \Lambda)
\end{equation}
for every $k \in \Z$. Indeed, this is proved in \cite[Theorem 3.1(a)]{GGM2} under the assumption that the symplectic form on the filling $W$ is aspherical. Without this assumption, the argument in the proof of \cite[Theorem 3.1(a)]{GGM2} applies word-for-word except for the nuance that in $\HC_*(W)$ we have to use coefficients in the universal Novikov field $\Lambda$ and the action filtration introduced by McLean and Ritter \cite{MR}; see Section \ref{sec:ESH}. (Note here that every periodic orbit of a connection form on $M$ has a unitary trivialization of the contact structure under which the linearized Reeb flow is complex linear \cite[Remark 8.8]{KvK} and this allows us to use Lemma \ref{lemma:lfh}.)

It follows from \eqref{eq:CH+} that
\[
\dim \HC^0_n(W) = \sum_{m\in\N} \dim H_{2n-2mc_B}(B; \Lambda)
\]
and 
\[
\dim \HC^0_{n+2jc_B}(W) = \sum_{m\in\N} \dim H_{2n+2(j-m)c_B}(B; \Lambda)
\]
for every $j \in \N$. Thus,
\begin{align*}
\dim \HC^0_{n+2jc_B}(W) & \geq \dim H_{2n}(B; \Lambda) + \sum_{m\in\N} \dim H_{2n-2mc_B}(B; \Lambda) \\
& > \dim \HC^0_n(W)
\end{align*}
for every $j \in \N$. When $\rho<0$ we have the relation
\begin{equation}
\label{eq:CH}
\dim \HC^0_k(W) = \sum_{m\in\N} \dim H_{k+2mc_B+n} (B; \Lambda)
\end{equation}
for every $k \in \Z$. (Notice here that there is a typo in the sign of $n$ in the isomorphism (3.3) of \cite{GGM2}.) This implies that, for all $j \in \N$,
\[
\dim \HC^0_{-n}(W) = \sum_{m\in\N} \dim H_{2mc_B}(B; \Lambda)
\]
and 
\[
\dim \HC^0_{-n-2jc_B}(W) = \sum_{m\in\N} \dim H_{2(m-j)c_B}(B; \Lambda).
\]
Hence,
\begin{align*}
\dim \HC^0_{-n-2jc_B}(W) & \geq \dim H_{0}(B; \Lambda) + \sum_{m\in\N} \dim H_{2mc_B}(B; \Lambda) \\
& > \dim \HC^0_{-n}(W)
\end{align*}
as desired.
\end{proof}

\section{Proof of Corollary \ref{cor:brieskorn}}
\label{sec:brieskorn}

Given $p\equiv 1\ \text{mod}\ 8$ and an even natural number $n$ let $\Sigma_p:=\Sigma_{(p,2,\dots,2)}$ be the corresponding Brieskorn sphere of dimension $2n+1$. It is proved in \cite{Ust} that $\Sigma_p$ admits a non-degenerate contact form $\alpha$ with finitely many simple closed orbits such that the Conley-Zehnder indices of the good periodic orbits have the same parity. It is well known that $\Sigma_p$ admits a symplectic filling given by a Liouville domain $W$ satisfying $c_1(TW)=0$. Therefore, using the fact that $\HC_*(W)$ is the homology of a complex generated by the good closed orbits of $\alpha$ and the computation of the indices of the orbits from \cite[Lemma 4.3]{Ust} we can conclude that
\[
\dim \HC_n(W) = 1\quad\text{and}\quad \dim \HC_{n+j(4+2p(n-1))}(W) = 2
\]
for every $j \in \N$. (Note here that in \cite{Ust} the contact homology degree is used and the dimension of the Brieskorn sphere is $2n-1$.) Therefore, $\Sigma_p$ satisfies the hypotheses of Theorem \ref{thm:main} and we conclude the proof of Corollary \ref{cor:brieskorn}.

\section{Proof of Theorem \ref{thm:sums}}
\label{sec:sums}

The following exact triangle can be found in \cite[Theorem 4.4]{BO17} or \cite[Corollary 9.26]{CO}.

\begin{theorem}
Let $W$ be a Liouville domain of dimension $2n+2$ and $W'$ be obtained from $W$ by attaching a subcritical handle of index $k<n+1$. Suppose that $c_1(TW)$ and $c_1(TW')$ vanish. Let $\sigma$ be a non-vanishing section of $\Lambda^{n+1}_\C TW$ and $\sigma'$ a non-vanishing section of $\Lambda^{n+1}_\C TW'$ extending $\sigma$. We then have an exact triangle
\[
\xymatrix{
\HC_*(D^{2(n+1-k)})\ar[rr]& & \HC_*(W')\ar[ld]\\
& \HC_*(W)\ar[ul]^{[-1]}&
}
\]
in which the map $\HC_*(W')\to \HC_*(W)$ is the transfer map and the gradings in $\HC_*(W)$ and $\HC_*(W')$ are induced by $\sigma$ and $\sigma'$ respectively.
\end{theorem}
\noindent The case that $k=1$ and $W=W_1\sqcup W_2$ corresponds to boundary connected sums.

Since $\HC_{n}(D^{2n})$ and $\HC_{n-1}(D^{2n})$ vanish, we have that
\[
	\dim \HC_n(W_1\# W_2) = \dim \HC_n(W_1) + \dim \HC_n(W_2).
\]
Now, suppose that $W_1$ and $W_2$ satisfy the hypothesis of Theorem \ref{thm:sums}, that is, there exist natural numbers $K_1$ and $K_2$ such that 
\[
\dim \HC_{n}(W_1) < \dim \HC_{n+jK_1}(W_1),\qquad \dim \HC_{n}(W_2) < \dim \HC_{n+jK_2}(W_2)
\] 
for every $j \in \N$, where the gradings in $\HC_*(W_1)$ and $\HC_*(W_2)$ are induced by $\sigma_1$ and $\sigma_2$ respectively. Let $K=\lcm\{K_1,K_2\}$. Since the dimension of $\HC_{*}(D^{2n})$ is at most one in any degree, it follows from the exact triangle that
\begin{align*}
\dim \HC_{n+jK}(W_1\# W_2) & \geq \dim \HC_{n+jK}(W_1) + \dim \HC_{n+jK}(W_2)-1 \\
&> \dim \HC_{n}(W_1) + \dim \HC_{n}(W_2) \\
&= \dim \HC_n(W_1\#W_2)
\end{align*}
for every $j \in \N$, where the grading in $\HC_*(W_1\# W_2)$ is induced by any non-vanishing section $\sigma$ which extends $\sigma_1$ and $\sigma_2$. Hence, under the additional assumption $c_1(T(W_1\# W_2))=0$, we conclude that $W_1\#W_2$ satisfies the hypothesis of Theorem \ref{thm:main}.

\end{document}